\documentclass[10pt]{article}
\usepackage{color}
\usepackage{amssymb}
\usepackage{amsthm,array,amssymb,amscd,amsfonts,latexsym, url}
\usepackage{amsmath}
\usepackage[all]{xy}
\newtheorem{theo}{Theorem}[section]
\newtheorem{prop}[theo]{Proposition}
\newtheorem{claim}[theo]{Claim}
\newtheorem{lemm}[theo]{Lemma}
\newtheorem{sublemm}[theo]{Sublemma}
\newtheorem{coro}[theo]{Corollary}
\newtheorem{rema}[theo]{Remark}
\newtheorem{Defi}[theo]{Definition}

\newtheorem{example}[theo]{Example}
\newtheorem{question}[theo]{Question}

\title{On   fibrations and measures of irrationality  of hyper-K\"{a}hler manifolds}
\author{Claire Voisin\footnote{The author is supported by the ERC Synergy Grant HyperK (Grant agreement No. 854361).}}

\date{}

\newfont{\gothic}{eufb10}
\begin{document}
\maketitle
\begin{abstract} We  prove some  results on the fibers and images of rational maps from   a hyper-K\"{a}hler manifold. We study in particular the minimal genus of fibers of  a fibration into curves. The last section of this paper is devoted to the study of the rational map defined by a   linear system on a hyper-K\"{a}hler fourfold satisfying numerical conditions similar to those considered by O'Grady in his study of  fourfolds numerically equivalent to $K3^{[2]}$. We  extend his results to this more general context.
 \end{abstract}
\section{Introduction}
The paper \cite{einlazMIR} introduced and discussed two  numerical birational invariants of a  projective variety $X$, the covering gonality ${\rm covgon}(X)$ and the degree of irrationality ${\rm irr}(X)$. The former is defined as
the minimal gonality of a  curve $C$, which is the general fiber of a family
$$\psi :\mathcal{C}\rightarrow B,\,\,\phi: \mathcal{C}\rightarrow X$$ of curves covering $X$, that is, $\phi $ is dominant and nonconstant on the fibers of $\psi$. The second number is defined as the minimal degree of a dominant rational map $X\dashrightarrow \mathbb{P}^n$, $n={\rm dim}\,X$.
Obviously, one has ${\rm irr}(X)\geq {\rm covgon}(X)$ but the inequality is strict in many cases. For example, the covering gonality of a uniruled manifold is $1$, while its irrationality is $1$ only if it is rational. One can similarly  introduce the ``covering genus'' ${\rm covgen}(X)$, namely the genus  of a  curve $C$, which is the general fiber of a family
$$\psi :\mathcal{C}\rightarrow B,\,\,\phi: \mathcal{C}\rightarrow X$$ of curves covering $X$.

There are several   similarly defined  numbers that can be studied, namely the ``fibering  gonality" ${\rm fibgon}(X) $ and the ``fibering genus''
${\rm fibgen}(X)$
defined as follows
\begin{Defi}\label{defifibgon} The fibering gonality of $X$ is the minimal gonality of the general fiber of a fibration $X\dashrightarrow B$ into curves.
 The fibering genus of $X$ is the minimal genus of the general fiber of a fibration $X\dashrightarrow B$ into curves.
\end{Defi}
Here, the  ``general fiber" of a rational map $\phi: X\dashrightarrow B$ is defined as  the general fiber of a resolution of singularities
$\tilde{\phi}:\widetilde{X}\rightarrow B$. Instead of studying coverings of $X$ by varieties of a given type, we thus  study fibrations, namely dominant rational map $X\dashrightarrow B$ with connected fibers and ${\rm dim}\,B< {\rm dim}\,X$, with fibers of a given type. There are   obvious inequalities
\begin{eqnarray} \label{ineqintrotrivial} {\rm covgon}(X)\leq {\rm fibgon}(X),\,\,{\rm covgen}(X)\leq {\rm fibgen}(X).
\end{eqnarray}
Another  simple  comparison  between the fibering genus  and the fibering gonality of  a projective variety $X$  introduced  in (\ref{defifibgon})  is
\begin{eqnarray} \label{eqineqpourintroBN}  {\rm fibgon}(X)\leq   \lceil{\frac{{\rm fibgen}(X)}{2}}\rceil+1.
\end{eqnarray}
which  follows indeed from the Brill-Noether theory showing the existence of $g_k^1$ on curves of genus $\leq 2k-2$.
Note that, in the case of a surface, the fibering genus is called the Konno invariant \cite{konno}. Ein and Lazarsfeld studied in {\it loc. cit.} a different  higher dimensional generalization of it in \cite{einlazkonno}, defined as the minimal geometric genus $p_g$ of a fiber of a   rational map to $\mathbb{P}^1$.

Ein and Lazarsfeld prove that the Konno invariant of a $K3$ surface with Picard group of rank  $1$ generated by a line bundle of self-intersection $h$ grows like $\sqrt{h}$. This is in strong contrast with the covering genus which is always equal to  $1$.
A beautiful  construction by Koll\'ar \cite{kollar} shows that a rationally connected smooth projective manifold, hence of covering gonality $1$ and covering genus $0$,  can be non-ruled, hence can have fiber gonality at least $2$ and fiber genus at least $1$, so both inequalities in (\ref{ineqintrotrivial}) are strict in general.

In the case of hyper-K\"{a}hler manifolds,    the following question  asked by Pacienza (oral communication) is still open.
\begin{question} \label{questionpacienza} Let $X$ be a hyper-K\"{a}hler manifold which is projective and very general in moduli. Is $X$ swept-out by elliptic curves? Equivalently, is ${\rm covgen}(X)=1$?
\end{question}
Here the assumptions on $X$ mean that $X$ is equiped with a given polarization (very ample line bundle) and, equipped with this polarization,  is very general in the corresponding moduli space of polarized hyper-K\"{a}hler manifolds. In particular, we have $\rho(X)=1$ by generalities on the period map. We expect that the answer to this question is no in some examples but were not able to prove or disprove it even on some explicit examples like the Fano variety of lines on a cubic fourfold, although we described in \cite{voisintriangle} some consequences of the existence of a covering by elliptic curves. Note that, if $\rho(X)=2$, the example of Hilbert schemes $S^{[n]}$ for any projective $K3$ surface shows that we may have many such coverings. Indeed, it is well-known that $S$ itself has many coverings by $1$-parameter families of elliptic curves $E_t$, and then $z\times E_t\subset S^{[n]}$ for any $0$-dimensional subscheme
$z\subset S$ of length $n-1$ not intersecting $E$ is an elliptic curve in $S^{[n]}$ and these elliptic curves  cover $S^{[n]}$.

In contrast, we will show  in Section \ref{secgenineq} that  Question \ref{questionpacienza} has an easy negative answer if the covering genus is replaced by the fibering genus:
\begin{prop} \label{theofibgen} Let $X$ be a hyper-K\"{a}hler manifold of dimension $2n$. Then if $n>1$, one has
\begin{eqnarray} \label{ineqfibgen} {\rm fibgen}(X)\geq 3\end{eqnarray}
\begin{eqnarray} \label{ineqfibgon} {\rm fibgon}(X)\geq 3\end{eqnarray}
\end{prop}
The proofs are  elementary. The inequality (\ref{ineqfibgen}) is a consequence of the inequality ${\rm fibgen}(X)\geq 2$ and of (\ref{ineqfibgon}). The inequality ${\rm fibgen}(X)\geq 2$ can be given several proofs. One of them  generalizes to the
 case of fibrations by varieties birational to abelian varieties for which we prove the following result.
\begin{theo} \label{theoabfib} Let $X$ be a hyper-K\"{a}hler manifold of dimension $2n$. Then if $X$ admits a fibration $X\dashrightarrow B$ with general fiber birational to an abelian variety of dimension $g$, one has $g= n$, hence also ${\rm dim}\,B= n$, and  the fibration is Lagrangian.
\end{theo}
 Theorem \ref{theoabfib} is  wrong if we replace ``fibrations" by ``coverings". A counterexample is given by the variety $S^{[n]}$ above and its coverings by elliptic curves.  In Section \ref{sectmaintheo}, we will give  examples with $\rho(X)=1$ of  a very general hyper-K\"{a}hler varieties of dimension $8$ swept-out by varieties birational to abelian surfaces.
Note that, if instead of studying rational maps, we consider  actual morphisms from $X$ to a smaller dimensional basis $B$, then we already  know they are quite restricted when $X$ is a hyper-K\"{a}hler manifold. Indeed, if $B$ is not a point, Matsushita \cite{matsushita1}, \cite{matsushita2} proves that they are given by Lagrangian fibrations and in particular the dimension of $B$ is $n$.

Concerning  the fibering genus, we will prove
\begin{theo} \label{theocurves} Let $X$ be a hyper-K\"{a}hler manifold of dimension $2n$ with  $n\geq 3$ and $b_2(X)_{\rm tr}\geq 5$. Assume that the Mumford-Tate group of the Hodge structure on $H^2(X,\mathbb{Q})_{\rm tr}$ is maximal. Then if $X$ admits a fibration $\phi: X\dashrightarrow B$, with ${\rm dim}\,B=2n-1$,   the general fiber  of $\phi$ has  genus $g\geq {\rm Inf}(n+2, 2^{\llcorner\frac{ b_{2,{\rm tr}}-3}{2}\lrcorner})$. In other words,
$${\rm fibgen}(X)\geq {\rm Inf}(n+2,2^{\llcorner\frac{ b_{2,{\rm tr}}-3}{2}\lrcorner}).$$
\end{theo}
Note that the bound in Theorem \ref{theocurves} is presumably not optimal. Looking at the proof, we see that a more natural bound would be ${\rm fibgen}(X)\geq {\rm Inf}(2n-1,2^{\llcorner\frac{ b_{2,{\rm tr}}-3}{2}\lrcorner})$. (This is also the reason for the assumption $n\geq 3$ in Theorem \ref{theocurves}.)
For $n=2$, we do not know what the correct bound is, but we can easily construct an example where the bound $g=n+2$ is achieved. Indeed, let $Y$ be  a smooth cubic fourfold, and let $Y_H\subset Y$ be a hyperplane section.
Let $X$ be the variety of lines of $Y$. It admits a rational map $$ X\dashrightarrow Y_H$$
which to a general point $\delta\in X$ parameterizing a  line $\Delta\subset Y$ associates the intersection point $y:=\Delta\cap Y_H\in Y_H$.
The fiber of this map over a general point $y\in Y_H$ is the curve of lines in $Y$ passing through $y$, and this is well-known (see \cite{clegri})  to be a genus $4$ curve, complete intersection of a quadric and a cubic in $\mathbb{P}^3$.

  Proposition  \ref{theofibgen}   and the example above leave  open the following

\begin{question}\label{question2} Are there  hyper-K\"{a}hler fourfolds   with ${\rm fibgen}=3$? Are there  hyper-K\"{a}hler sixfolds   with ${\rm fibgen}=5$?
\end{question}

We now turn to the measure of irrationality ${\rm irr}\,X$  mentioned at the beginning of this introduction. In the geometric  context we are considering, namely hyper-K\"{a}hler manifolds, which in any case are not rational, there are two   natural variants of this number, namely
\begin{eqnarray}\label{eqRCirr} {\rm RCirr}(X):={\rm Inf}\,{\rm deg}\,\phi,
\end{eqnarray}
where  $\phi$ runs through all the generically finite rational maps $X\dashrightarrow Y$, with $Y$ smooth projective  rationally connected and
\begin{eqnarray}\label{eqcohirr} {\rm cohirr}(X):={\rm Inf}\,{\rm deg}\,\phi,
\end{eqnarray}
where    $\phi$ runs through all the generically finite rational maps $X\dashrightarrow Y$, with $Y$ smooth projective  with $H^0(Y,\Omega_Y^l)=0$ for $l>0$.
\begin{rema}{\rm  When $X$ is a hyper-K\"{a}hler fourfold, it is equivalent in (\ref{eqcohirr}) to consider the   smooth projective varieties $Y$ with $H^0(Y,K_Y)=0$, since the existence of a dominant generically finite  rational map $\phi: X\dashrightarrow Y$ then implies that  $H^0(Y,\Omega_Y^l)=0$ for $l>0$. Indeed, if $Y$ has a holomorphic $2$-form, it is generically nondegenerate since it pulls-back to the holomorphic $2$-form on $X$, hence $h^0(Y,K_Y)\not=0$.}
\end{rema}
Obviously ${\rm cohirr}(X)\leq  {\rm RCirr}(X)\leq {\rm irr}(X)$. The invariant ${\rm cohirr}(X)$ has been studied  by Alzati and Pirola  in \cite{alzatpirola}.
A particular case of their results is
\begin{theo}\label{theomir} Let $X$ be a hyper-K\"{a}hler manifold of dimension $2n$.  Then
${\rm cohirr}(X)\geq n+1$. In particular, if ${\rm dim}\,X\geq 4$, one has ${\rm cohirr}(X)\geq 3$. If ${\rm dim}\,X\geq 6$, one has ${\rm cohirr}(X)\geq 4$.
\end{theo}
Combining Theorem \ref{theomir} and Theorem \ref{theocurves}, we will prove in Section \ref{secgenineq}
\begin{coro}\label{corotardif} Let  $X$ be a hyper-K\"{a}hler manifold of dimension $\geq 6$. Assume that $b_2(X)_{\rm tr}\geq 9$ and $X$ is very general with given Picard number. Then ${\rm fibgon}(X)\geq 4$.
\end{coro}

  It is likely that a better  lower bound for  Theorem \ref{theomir}, maybe depending on numerical data as in  \cite{martino},  which studies the case of abelian surfaces.  In  the case of dimension $4$, we leave this  as
\begin{question} \label{question1}  Let $X$ be a hyper-K\"{a}hler fourfold which is very general  with fixed Picard number. Is it true that  ${\rm cohirr}(X)\geq 4$?
\end{question}
We prove one result in this direction in Section \ref{secmir}, namely Proposition \ref{prodeg3} which is used in the  last section of the paper.    We  establish there a generalization of a result  of O'Grady (see \cite{ogradyccm} or Theorem \ref{theoogadry}). O'Grady studies the rational map $\phi_L:X\dashrightarrow\mathbb{P}^5$ induced by the complete linear system $|L|$, for   a line bundle $L$ of top self-intersection $12$ on a compact K\"{a}hler fourfold $X$ which is numerically equivalent to $K3^{[2]}$. Assuming $X$ is very general with Picard number $1$, O'Grady proves that the image of $\phi_L$ is a hypersurface of degree $\geq 6$.
We prove a similar result (see Theorem  \ref{propLMnef3}) under different  assumptions. First of all, $X$ is only known to have the same Betti numbers, Chern numbers, and Fujiki constant as a hyper-K\"{a}hler fourfold of type $K3^{[2]}$. Second, in our case, the line bundle is the sum $L+M$, where both $L$ and $M$ are numerically effective and satisfy the intersection conditions
\begin{eqnarray}\label{eqincon}L^4=0,\,M^4=0\,\,L^2M^2=2,\end{eqnarray}
which implies $(L+M)^4=12$.
Our result is
\begin{theo}\label{theomerdique}  Under the assumptions above, assuming $X$ is very general with Picard number $2$ and $h^0(X,L)=0$,
the image of $\phi_{L+M}: X\dashrightarrow\mathbb{P}^5$ is not rationally connected.
\end{theo}
Although this result may seem  a bit specific, this statement  is  needed in order  to conclude the proof of the main result
in \cite{DHMV}, namely that a hyper-K\"{a}hler fourfold  $X$ admitting two integral degree $2$ cohomology classes $l,\,m$ satisfying the condition (\ref{eqincon})  has to be of $K3^{[2]}$ deformation type.

Theorem \ref{theomerdique} is proved by a case-by-case analysis. As will be clear from the proof, a positive answer to Question \ref{question1} and a negative answer to Question \ref{question2} would  greatly simplify the   proof, since   by   Lemmas \ref{leeasyA}, \ref{leeasyB}, \ref{leeasyC}, \ref{leeasyD}, \ref{le3sur3dim4} and Claim \ref{claimpourle63})  the most difficult cases to exclude  are those where    $X$ is fibered into curves of genus $3$, or $\phi_{L+M}$ has degree $3$ on its image.

\vspace{0.5cm}

{\bf Thanks.} {\it I am grateful to the scientific committee of the   2021 Mathematical Congress of the Americas  for inviting me to  this beautiful conference. I thank Ciro Ciliberto and Chritian Peskine for their help, in particular  for the proof of Proposition \ref{proquatrecinqtrois}, and Olivier Martin for  mentioning the work of Alzati and Pirola.  I thank the referee for his/her wonderful refereeing work and for numerous  constructive comments and  suggestions.}

\section{Fibrations of hyper-K\"{a}hler manifolds   by curves and  abelian varieties \label{sectmaintheo}}
\subsection{Some general inequalities\label{secgenineq}}
We start by establishing easy lower bounds for the fibering genus and  gonality, and  various  irrationality  invariants of hyper-K\"{a}hler manifolds.
\begin{lemm} \label{proinvol} (see also \cite{ogradyccm}) Let $X$ be a hyper-K\"ahler manifold of dimension $2n$. If there exists a dominant  rational map
$\phi: X\dashrightarrow Y$  of degree $2$, where $Y$ is a  smooth projective variety, then
 $h^0(Y,\Omega_Y^{4k})=1$ for $k\leq n$. In particular, the cohomological  measure of irrationality
${\rm cohirr}(X)$  of a hyper-K\"{a}hler $2n$-fold   is strictly greater than $2$ for  $n\geq2$.
\end{lemm}
\begin{proof} We observe that,  as $\phi$ has degree $2$, there is a rational involution
$\iota$  on $X$ over $Y$. As $h^0(X,\Omega_X^2)=1$, one has $\iota^*\sigma_X=\pm\sigma_X$.  It follows that $\iota^*\sigma_X^{2k}=\sigma_X^{2k}$ for any integer $k$. Thus the $(4k,0)$-form $\sigma_X^{2k}$ on $X$, which is nonzero for $2k\leq n$,  descends to $Y$, proving the inequality $h^0(Y,\Omega_Y^{4k})\geq1$ for $k\leq n$. The inequality  $h^0(Y,\Omega_Y^{4k})\leq1$ for $k\leq n$ follows from the fact that $\phi^*$ is injective on holomorphic forms.
\end{proof}

We now  apply this result to the proof of Proposition \ref{theofibgen}.
\begin{proof}[Proof of Proposition  \ref{theofibgen}]
 We first prove
 \begin{lemm}\label{leelliptic}  Let $X$ be a hyper-K\"{a}hler $2n$-fold with $n\geq 2$. Then $X$ does not admit a fibration
 $\phi:X\dashrightarrow Y$ into elliptic curves, hence ${\rm fibgen}(X)\geq2$.
 \end{lemm}
 \begin{proof} Let $\tau:\widetilde{X}\rightarrow X$, $\tilde{\phi}:\widetilde{X}\rightarrow Y$ be a resolution of the indeterminacies of $\phi$, with $\widetilde{X}$ smooth. Then, as the general  fiber $F$ of $\tilde{\phi}$ is elliptic, one has $K_{\widetilde{X}\mid F}=\mathcal{O}_F$. But $K_{\widetilde{X}}$ has a section, whose divisor has for support the exceptional divisor of $\tau$. It follows that $F$ does not intersect the exceptional divisor of $\tau$. In other words, $\phi$ is quasiholomorphic. This contradicts a theorem  of Matsushita \cite{matsushita1} which says that a quasiholomorphic map from a hyper-K\"ahler $2n$-fold to a manifold of smaller dimension has image of dimension $\leq n$.
 \end{proof}
  Inequality (\ref{ineqfibgon}) in Proposition \ref{theofibgen}  implies inequality (\ref{ineqfibgen}) since curves of genus $\leq 2$  have gonality $\leq2$.  We now prove the inequality (\ref{ineqfibgon}). Assume  that $X$ admits a fibration $\phi: X\dashrightarrow Y$ into hyperelliptic curves.  By Lemma \ref{leelliptic}, the fibers have genus  at least $2$. The smooth projective variety    $Y$ obviously  satisfies $h^0(Y,\Omega_Y^l)=0$ for $l>0$. Furthermore there exists a relative hyperelliptic involution
$\iota$ on $X$ such that any smooth model  $Q$ of $X/\iota$ is a fibration into $\mathbb{P}^1$ over $Y$. Thus  $Q$ satisfies
$h^0(Q,\Omega_Q^4)=0$. This contradicts Lemma \ref{proinvol}.
\end{proof}

Another easy result is the following

\begin{lemm} \label{lecovgenRCirr} Let $X$ be a hyper-K\"{a}hler manifold of dimension $\geq 4$. Then
\begin{eqnarray}\label{eqpourrcirrHK}  {\rm RCirr}(X)\leq 2{\rm fibgen}(X)-2.
\end{eqnarray}
\end{lemm}
\begin{proof} Let $f: X\dashrightarrow B$ be a fibration realizing the fibering genus, so that the fibers have genus $g={\rm fibgen}(X)$, and $\tilde{f}:\widetilde{X}\rightarrow B$ a resolution of the indeterminacies of $f$.  By Lemma \ref{leelliptic}, we know  that  $g\geq 2$. By \cite{lin}, the base $B$ is rationally connected. We now choose a rank $2$ subsheaf $\mathcal{F}$ of the sheaf
$R^0\tilde{f}_*K_{\widetilde{X}/B}$. The variety $\mathbb{P}(\mathcal{F})$ is generically a $\mathbb{P}^1$-bundle over $B$, hence is rationally connected, and there is a natural rational map
$$\psi:\widetilde{X}\dashrightarrow  \mathbb{P}(\mathcal{F})$$
over $B$, which is of degree $\leq 2g-2$.
\end{proof}
\begin{rema}{\rm  By Theorem \ref{theomir},  Lemma \ref{lecovgenRCirr}  implies that, for any hyper-K\"{a}hler fourfold  $X$, one has $ {\rm fibgen}(X)\geq \frac{n+3}{2}$, a statement to be compared with Theorem \ref{theocurves}.}
\end{rema}
We finally combine the  results above to prove
\begin{prop}\label{theocomp}   Let $X$ be a projective hyper-K\"{a}hler manifold of dimension $\geq4$.  Assume that  the fibering gonality of $X$ is $3$. Then one of the following possibilities holds:

(i)  ${\rm fibgen}(X)=3$ and ${\rm RCirr}(X)\leq 4$.

(ii)  ${\rm fibgen}(X)=4$ and ${\rm RCirr}(X)\leq 6$.

(iii)  ${\rm fibgen}(X)>4$ and ${\rm RCirr}(X)= 3$.
\end{prop}
\begin{proof}  Let $\phi: X\dashrightarrow B$ be a fibration realizing the fibering gonality, so that the fibers of $f$ are trigonal curves. We know by Proposition \ref{theofibgen} that the genus of the fibers is at least $3$.  If the genus of the fibers is $3$ or $4$, then we apply Lemma \ref{lecovgenRCirr}  and get the inequalities  in (i) and (ii). If the genus of the fibers is $\geq 5$, we recall that a curve of genus $\geq 5$  which is trigonal admits a unique $g_3^1$, unless it is hyperelliptic, which is excluded by  Proposition   \ref{theofibgen}. It follows that there exists
a fibration $P\dashrightarrow B$ into $\mathbb{P}^1$'s and a rational map of degree $3$
$$ \psi: X\dashrightarrow P
$$
over $B$, which induces the trigonal map on the fibers of $f$. As $B$ is rationally connected, $P$ is rationally connected and  thus ${\rm RCirr}(X)= 3$.
\end{proof}
\subsection{Proof of Theorem \ref{theoabfib}}
This section is devoted to the proof of Theorem \ref{theoabfib}.  Let $X$ be a hyper-K\"{a}hler manifold of dimension $2n$ admitting a fibration $f:X\dashrightarrow B$ with general fiber birational to an abelian variety of dimension $g$.
Let $L$ be an ample line bundle on $X$. The  restriction to the general fiber $\widetilde{X}_t$ of a resolution
$\tilde{f}:\widetilde{X}\rightarrow B$ of the indeterminacies of $f$ has top-self-intersection $D:={\rm deg}\,L^g_{\mid \widetilde{X}_t}$. We will denote by $Z_t$ the $0$-cycle $L^g_{\mid \widetilde{X}_t}\in {\rm CH}_0( \widetilde{X}_t)$.

As $\widetilde{X}_t$ is birational to its Albanese variety, there is a rational action by translation
$$\widetilde{X}_t\times {\rm Alb}\,\widetilde{X}_t\dashrightarrow \widetilde{X}_t$$
$$ (x,u)\mapsto x+u $$
of ${\rm Alb}\,\widetilde{X}_t$ on $\widetilde{X}_t$.

For any integer $k$, we can construct a rational self-map
\begin{eqnarray} \label{eqselmap} \Psi_k: X\dashrightarrow X
\end{eqnarray}
preserving $f$, that is, acting fiberwise, and defined by
\begin{eqnarray} \label{eqforlmulapsik} \Psi_k(x)=x+k\,{\rm alb}_{\widetilde{X}_t}(Dx-Z_t),\,x\in \widetilde{X}_t.
\end{eqnarray}
\begin{lemm}\label{lepourdegfib} The degree of $\Psi_k$ is $(kD+1)^{2g}$.
\end{lemm}
\begin{proof} As $\Psi_k$ acts in  a fiberwise way with respect to $f$, its degree is equal to the degree of its  restriction to the fibers $\widetilde{X}_t$. By (\ref{eqforlmulapsik}), this restriction is  birationally conjugate to the multiplication by $kD+1$ on a $g$-dimensional abelian variety, which proves the result.
\end{proof}
We next have
\begin{lemm} \label{lepourmu}  Let $\sigma_X\in H^0(X,\Omega_X^2)$ be a generator. We have either $\Psi_k^*\sigma_X=(kD+1)\sigma_X$ or
$\Psi_k^*\sigma_X=(kD+1)^2\sigma_X$. In the first case, the fibers $\widetilde{X}_t$ are isotropic for $\sigma_X$.
\end{lemm}
\begin{proof} As $\Psi_k^*\sigma_X$ is a nonzero  holomorphic $2$-form on $X$, it must be a nonzero multiple of $\sigma_X$, so $\Psi_k^*\sigma_X=\mu \sigma_X$. As $\Psi_k$ acts in a fiberwise way, we have
\begin{eqnarray} \label{eqactpsion2} (\Psi_k^*\sigma_X)_{\mid \widetilde{X}_t}=\Psi_{k\mid \widetilde{X}_t}^*(\sigma_{X\mid \widetilde{X}_t}).\end{eqnarray}
As $\Psi_{k\mid \widetilde{X}_t}^*$ acts as multiplication by $(kD+1)^2$ on the transcendental  degree $2$ cohomology of
$\widetilde{X}_t$, (\ref{eqactpsion2}) implies that $\mu=(kD+1)^2$ if the fibers $\widetilde{X}_t$ are not isotropic for $\sigma_X$.
If the fibers $\widetilde{X}_t$ are  isotropic for $\sigma_X$, then $\sigma_{X}$, (or rather its pull-back  $\tau^*\sigma_{X}$ on a model  $\widetilde{X}$ where $f$ is well-defined) maps to an  element $\sigma_t$ of $H^0(\widetilde{X}_t,\Omega_{\widetilde{X}_t})\otimes \Omega_{B,t}$ which is nonzero for generic $t$ as otherwise $\tau^*\sigma_X$ would be everywhere degenerate.
As $\Psi_{k\mid \widetilde{X}_t}^*$ acts as multiplication by $kD+1$ on $1$-forms on
$\widetilde{X}_t$, we get in this case $\mu=kD+1$, using the fact
that
 the action of $\Psi_k^*$ on $\sigma_t$ is induced by the action of $\Psi_{k\mid \widetilde{X}_t}^*$ on the space $H^0(\widetilde{X}_t,\Omega_{\widetilde{X}_t})$.
\end{proof}
\begin{proof}[Proof of Theorem \ref{theoabfib}] We have $\Psi_k^*\sigma_X=\mu \sigma_X$, which we write in the form
\begin{eqnarray} \label{eqpulldes} \widetilde{\Psi}_k^*(\sigma_X)=\mu \tau^*\sigma_X,
\end{eqnarray}
where
$$\tau:\widetilde{X}\rightarrow X,\,\, \widetilde{\Psi}_k: \widetilde{X}\rightarrow X$$
is a desingularization of $\Psi_k: X\dashrightarrow X$.
As we are now working with morphisms in (\ref{eqpulldes}) and $\mu$ is  a real number by Lemma \ref{lepourmu}, it follows that
$$\widetilde{\Psi}_k^*(\sigma_X^n\wedge \overline{\sigma_X}^n)=\mu^{2n}\tau^*(\sigma_X^n\wedge \overline{\sigma_X}^n).$$
Integrating both sides over $ \widetilde{X}$, we get ${\rm deg}\, \widetilde{\Psi}_k={\rm deg}\,\Psi_k=\mu^{2n}$.
By Lemma \ref{lepourdegfib}, we deduce that
\begin{eqnarray}\label{eqrelaintdegre} \mu^{2n}=(kD+1)^{2g},\end{eqnarray}
while by Lemma \ref{lepourmu}, we have $\mu=kD+1$ or $\mu=(kD+1)^2$.
If $\mu=(kD+1)^2$, we get by (\ref{eqrelaintdegre}) that $4n=2g$ which contradicts the fact that $g<2n$. Hence $\mu=kD+1$, which implies by  (\ref{eqrelaintdegre}) that $n=g$. Furthermore, the fibers are isotropic  in this case by Lemma \ref{lepourmu}.
\end{proof}

If instead of a fibration we consider a covering by varieties birational to abelian varieties of dimension $g$, we can conclude that they are isotropic, assuming that the Mumford-Tate group of the Hodge structure on $H^2(X,\mathbb{Q})_{\rm tr}$ is maximal and
 \begin{eqnarray} \label{eqineqKS} g<2^{\llcorner\frac{ b_{2,{\rm tr}}-3}{2}\lrcorner},
 \end{eqnarray}
 by applying the result in \cite{charles} (or \cite{vGvoisin} if $b_2(X)_{\rm tr}\geq 5$). Indeed, these results say that
 the Hodge structure on $H^2( X,\mathbb{Q})_{\rm tr} $,  which is simple, cannot be realized as a Hodge substructure of $H^2(A)$ for any abelian variety of dimension $g$ if $g$ satisfies (\ref{eqineqKS}).
 Note that, without the inequality (\ref{eqineqKS}), one can construct coverings by abelian subvarieties which are not isotropic, as shows the example of the generalized Kummer $K_n(A)$ which is swept out by
 copies of surfaces birational to  $A$.

 Concerning the statement about the dimension, the following is an example of   a covering of a hyper-K\"{a}hler manifold of dimension $8$ with $\rho=1$ by varieties birational to abelian surfaces.
 \begin{example}{\rm  Let $Y$ be a cubic fourfold, and let $X$ be the LLSvS $8$-fold of $Y$ (see \cite{LLSvS}). This is a $8$-fold which is deformation equivalent to ${\rm K3}^{[4]}$ (see \cite{addl}). Furthermore, if $Y$ is very general, one has $\rho(X)=1$. Let $F_1(Y)$ be the variety of lines in $Y$. There exists a dominant rational map (see \cite{voisinisotropic})
 $$\psi: F_1(Y)\times F_1(Y)\dashrightarrow X.$$
 Next, the hyper-K\"{a}hler manifold $F_1(Y) $ is itself covered  by surfaces birational to abelian surfaces. Indeed, consider the surfaces of lines $\Sigma_{Y_H}$ contained in a hyperplane section $Y_H$ of $Y$.  It is a classical fact that, when $Y_H$ has one singular point $y$, $\Sigma_H$ is birational to the symmetric product $C_{y,H}^{(2)}$, where $C_{y,H}$ is the curve of lines contained in $Y_H$ and passing through $y$. This curve is of genus $4$ when $Y_H$ has one ordinary quadratic singularity at $y$ and is smooth otherwise. When $Y_H$ has two more singular points $y'$ and $y''$, the curve
 $C_{y,H}$ becomes singular at these points, and its geometric  genus decreases to $2$. It is clear that $F_1(Y)$ is covered by these surfaces $\Sigma_{y,H}$  birational to the symmetric product $C_{y,H}^{(2)}$ of a curve of genus $2$, hence to abelian surfaces, and using the morphism $\psi$, it follows that
 $X$ is covered by the surfaces $\psi(x\times \Sigma_{y,H})$, which are  birational to abelian surfaces.}
 \end{example}
\subsection{Proof of Theorem \ref{theocurves}}
Let $X$ be a hyper-K\"{a}hler $2n$-fold and \begin{eqnarray}
\label{eqnomdenomlenom} f:\widetilde{X}\rightarrow B,\,\tau:\widetilde{X}\rightarrow X,
\end{eqnarray}  where $\tau $ is birational and $\widetilde{X}$ is smooth projective,  be a fibration into curves of genus $g$ over a base $B$ of dimension $2n-1$.
We have $h^0(B,\Omega_B^l)=0$ for any $l>0$ and in fact $B$ is rationally connected (see \cite{lin}). Let $b\in B$ be a general point so that the fiber $\widetilde{X}_b$ is smooth.
Consider the  natural morphism
\begin{eqnarray}\label{eqmorphismfiberint} \sigma_b: T_{B,b}\rightarrow H^0(\widetilde{X}_b,\Omega_{\widetilde{X}_b})
\end{eqnarray}
induced  by the vector bundle  morphism
$$T_{\widetilde{X}_b}\rightarrow f^*\Omega_{B,b}$$
defined by contraction with   the holomorphic $2$-form $\tau^*\sigma_X$  along $\widetilde{X}_b$.
\begin{lemm}\label{lerangdusigmab} The morphism $\sigma_b$ has rank $\geq n$.
\end{lemm}
\begin{proof}  Over the open set $B^0$ of $B$ of regular values of $f$, we have the relative Albanese fibration (or Jacobian)
$J_f\rightarrow B^0$. The $(2,0)$-form $\sigma_X$  on $\widetilde{X}^0$ induces a $(2,0)$-form
$$\sigma_J:=\mathcal{P}^*\sigma_X$$
on $J_f$, where   $\mathcal{P}\subset J_f \times_{B^0}\widetilde{X}^0 $ is a universal divisor, satisfying the assumption that, for some nonzero integer $d$,
\begin{eqnarray} {\rm alb}_{\widetilde{X}_b }(\mathcal{P}_y)=d y
\end{eqnarray}
for any $y\in J_{f,b}={\rm Pic}^0(\widetilde{X}_b)$.

We also have the Albanese embedding  (up to isogeny)
$${\rm alb}_f:  \widetilde{X}^0\rightarrow J_f,$$
which maps $x\in   \widetilde{X}_t$ to ${\rm alb}_{\widetilde{X}_t}((2g-2)x-K_{\widetilde{X}_t})$.
We have
\begin{eqnarray}  \label{eqpourplusclair} {\rm alb}_f^*\sigma_J=d(2g-2)\sigma_{\widetilde{X}}\end{eqnarray} since by definition of $\sigma_J$,
${\rm alb}_f^*\sigma_J=\Gamma^*\sigma_X$, where   $\Gamma$ is  the self-correspondence
$$ x\mapsto d((2g-2)x-K_{\widetilde{X}_t}),\,\,t=f(x)$$
of $X$ over $B$,
which induces multiplication by $d(2g-2)$ on ${\rm CH}_0(X)_{{\rm hom}}$ because $B$ is rationally connected.
It follows from (\ref{eqpourplusclair}) that we have the inequality   of generic ranks

$$ {\rm rank}\,\sigma_J \geq {\rm rank}\,\sigma_{\widetilde{X}},$$
 that is,
 \begin{eqnarray} \label{eqpourrankJ}   {\rm rank}\,\sigma_J \geq2n.
 \end{eqnarray}
By construction, the $(2,0)$-form $\sigma_J$ vanishes identically on the fibers $J_b=J(\widetilde{X}_b)$ of $\pi:J\rightarrow B^0$, hence induces a contraction map $\sigma_{J,b}: T_{B^0,b}\rightarrow H^0(J_b,\Omega_{J_b})$,
and, by (\ref{eqpourplusclair}), we clearly have a commutative diagram
$$\begin{matrix} &T_{B^0,b}&\stackrel{\sigma_{J,b}}{\rightarrow}& {\rm alb}_{\widetilde{X}_b }^* \Omega_{J_b}\\
&\parallel & &a\downarrow\\
&T_{B^0,b}&\stackrel{\sigma_{b}}{\rightarrow} &\Omega_{\widetilde{X}_b}
\end{matrix}
$$
of morphisms of vector bundles on $\widetilde{X}_b$, where $a:=\frac{1}{d(2g-2)}{\rm alb}_{\widetilde{X}_b }^*$.
Taking  global sections, we get

\begin{eqnarray}\label{eqidentCJ} \begin{matrix} &T_{B^0,b}&\stackrel{\sigma_{J,b}}{\rightarrow}& H^0(J_b,\Omega_{J_b})\\
&\parallel & &a\downarrow \\
&T_{B^0,b}&\stackrel{\sigma_{b}}{\rightarrow} &H^0(\widetilde{X}_b,\Omega_{\widetilde{X}_b}),
\end{matrix}
\end{eqnarray}
where the second vertical map is an isomorphism.
We now have
\begin{claim}  We have the equality of  rank along $J_b$
\begin{eqnarray}\label{eqranfibjac} {\rm rank}\, \sigma_J=2{\rm rank}\,\sigma_{J,b}
\end{eqnarray}
\end{claim}
\begin{proof} The torsion points of $J_b$ are dense in $J_b$ for the Zariski or Euclidean topology, so it suffices to prove the equality at a torsion point $y\in J_b$. Through such a point, there is a torsion multisection $Z_y\subset J$, which is transverse to the fiber $J_b$. The $(2,0)$-form $\sigma_J$ vanishes on $Z_y$, because torsion points are rationally equivalent  (up to torsion)  in the fibers to the origin $0_b\in J_b$ and all points in  the $0$-section  are rationally equivalent in $\widetilde{X}$ since the base $B$ is rationally connected.  It follows that  the matrix of $\sigma_{J}$ at $y$ in a basis adapted to the decomposition $T_{J,y}=T_{Z_y,y}\oplus T_{J_b,y}$, where $T_{Z_y,y}\cong T_{B,b}$,  takes the block  form
$$\begin{pmatrix} 0  &M_{\sigma_{J,b}}\\
-^t{M_{\sigma_{J,b}}}& 0
\end{pmatrix}
$$
where  $M_{\sigma_{J,b}}$ is the matrix of  $\sigma_{J,b}$.
\end{proof}

Using the identifications   (\ref{eqidentCJ}),  Lemma \ref{lerangdusigmab}  follows from (\ref{eqpourrankJ})  and  (\ref{eqranfibjac}).
\end{proof}
\begin{coro} \label{corgn} One has $g\geq n$.
\end{coro}
\begin{proof} Indeed, this follows from Lemma \ref{lerangdusigmab} since the rank of $\sigma_b$ is non greater than $g={\rm dim}\,H^0(\widetilde{X}_b,\Omega_{\widetilde{X}_b})$.
\end{proof}
\begin{rema}{\rm  For $n=2$, this gives a third proof of Lemma \ref{leelliptic}.}
\end{rema}
Let $\overline{\nabla}_b:T_{B,b}\rightarrow {\rm  Hom}\,(H^0(\widetilde{X}_b,\Omega_{\widetilde{X}_b}),H^1(\widetilde{X}_b,\mathcal{O}_{\widetilde{X}_b}))$ be the infinitesimal variation of Hodge structure of the  family of curves (\ref{eqnomdenomlenom}) at $b$.
We will use the following  classical symmetry result due to Donagi and Markman  \cite{donagimarkman} (see also \cite{bai}).
\begin{lemm}\label{lesymvhs} The bilinear map $T_{B,b}\otimes T_{B,b}\rightarrow H^1(\widetilde{X}_b,\mathcal{O}_{\widetilde{X}_b})$,
$$ (u,v)\mapsto \overline{\nabla}_u(\sigma_b(v))$$
is symmetric in $u$ and $v$.
\end{lemm}

\begin{proof}[Proof of Theorem \ref{theocurves}] In the situation above, assume that  $n\geq 3$ and $g= n$ or $g=n+1$. Then $2n-1>n+1\geq g$. It follows that,  at a general point $b\in B$,  the morphism $\sigma_b$  has a nontrivial kernel $K_b\subset T_{B,b}$. Moreover,
by Lemma \ref{lerangdusigmab}, the morphism $\sigma_b$ is either surjective or   has for image  a hyperplane in $H^0(\widetilde{X}_b,K_{\widetilde{X}_b})$.

 \vspace{0.5cm}

 {\it Case $g=n$ or $n+1$ and  $\sigma_b$ is  surjective.}
We first prove
\begin{lemm} \label{leKb} The kernel $K_b$ of $\sigma_b$ is contained in  the kernel of the Kodaira-Spencer  map
$\rho_b:T_{B,b}\rightarrow H^1(\widetilde{X}_b,T_{\widetilde{X}_b})$.
\end{lemm}
\begin{proof}  We apply Lemma \ref{lesymvhs}. It thus follows that
for $u\in K_b$, and any $v\in T_{B,b}$, we have
\begin{eqnarray}\label{equationapressym}  \overline{\nabla}_u(\sigma_b(v))=  \overline{\nabla}_v(\sigma_b(u))=0.
\end{eqnarray}
As $\sigma_b$ is surjective, this implies that
$\overline{\nabla}_u:H^0(\widetilde{X}_b,K_{\widetilde{X}_b})\rightarrow H^1( \widetilde{X}_b,\mathcal{O}_{\widetilde{X}_b})$ is identically $0$. However, we know by Proposition  \ref{theofibgen} that the fibers $\widetilde{X}_b$ are not hyperelliptic, hence
the map $$H^1(\widetilde{X}_b,T_{\widetilde{X}_b})\rightarrow {\rm Hom}\,(H^0(\widetilde{X}_b,K_{\widetilde{X}_b}),H^1( \widetilde{X}_b,\mathcal{O}_{\widetilde{X}_b}))$$
is injective. Hence $\rho_b(u)=0$.
\end{proof}

Let $m:B\dashrightarrow \mathcal{M}_g$ be the moduli map, which to a general point $b\in B$ associates the isomorphism class of the curve
$\widetilde{X}_b$. By Lemma \ref{leKb}, the vector space $K_b$ is tangent to the fiber of $m$, hence it follows that
the map $m$ has positive dimensional fibers.
We thus have, after Stein factorization, a fibration $m': B\dashrightarrow B'$ with connected positive dimensional  fibers, having the property that, restricted to a general  fiber of $m'$, the fibration $f$ becomes isotrivial.
Denoting $f':\widetilde{X}\dashrightarrow B'$ the composition $m'\circ f$, we can assume by modifying $\widetilde{X}$ that $f'$ is a morphism, and   prove
\begin{lemm}\label{levgeemen}  Assume that $X$ is very general with fixed Picard number, that $b_2(X)_{\rm tr}\geq 5$ and that
$g< 2^{\llcorner\frac{ b_{2,{\rm tr}}-3}{2}\lrcorner}$. Then the general fiber of $f'$ is isotropic for $\tau^*\sigma_X$.
\end{lemm}
\begin{rema}\label{remaquivaservir} {\rm   Lemma \ref{levgeemen}  says in particular  that ${\rm Ker}\,\rho_b\subset K_b$, hence  ${\rm Ker}\,\rho_b= K_b$ by Lemma \ref{leKb}. In particular ${\rm dim}\,B'=g$.}
\end{rema}

\begin{proof}[Proof of Lemma \ref{levgeemen}] As the fibration $f$ is isotrivial after restriction to the general fiber $B_{b'}\subset B$ of $m'$, the fiber
$\widetilde{X}_{b'}:={f'}^{-1}(b')$ is rationally dominated by a product $C_{b'}\times \widetilde{B_{b'}}$ where  $\widetilde{B_{b'}}$ is a generically finite cover of
$B_{b'}$ and $C_{b'}$ is isomorphic to the fibers of the restricted family, so in particular has genus $g$.  The fact that $X$ is very general with fixed Picard group implies that the Mumford-Tate group of the Hodge structure on $H^2(X,\mathbb{Q})_{\rm tr}$ is the orthogonal group of the Beauville-Bogomolov form, and as proved in \cite{vGvoisin}, this implies that, if the composite map
$$H^2(X,\mathbb{Q})_{\rm tr}\rightarrow H^2(\widetilde{X}_{b'},\mathbb{Q})\rightarrow H^1(C_{b'},\mathbb{Q})\otimes H^1( \widetilde{B_{b'}},\mathbb{Q})$$
is nontrivial, then the Hodge structure on  $H^1(C_{b'},\mathbb{Q})$ contains a simple factor of the Kuga-Satake weight $1$ Hodge structure of  $H^2(X,\mathbb{Q})_{\rm tr}$, hence in particular $g\geq 2^{\llcorner\frac{ b_{2,{\rm tr}}-3}{2}\lrcorner}$. This is excluded by assumption and it follows that
the    form $\tau^*\sigma_{X\mid \widetilde{X}_{b'}}$  is either $0$, or the pull-back of a holomorphic $2$-form $\tau_{b'}$  on the fiber $B_{b'}$.
In the first case, the lemma is proved. In the second case, there is a nonzero locally constant holomorphic $2$-form
$\eta_{b'}\in H^{2,0}(B_{b'})$ whose pull-back to $\widetilde{X}_{b'}$ is $\tau^*\sigma_{X\mid \widetilde{X}_{b'}}$ and Deligne's global invariant cycle theorem then implies that there is a holomorphic $2$-form
$\eta$ on $B$  whose restriction to $B_{b'}$ is $\eta_{b'}$. This is impossible since otherwise $f^*\eta$ would provide a  nonzero holomorphic $2$-form on $X$ of rank $<2n$.
\end{proof}

Let  $B'_0$ be the Zariski open set of $B'$ over which the morphism $f':\widetilde{X}\rightarrow B'$ is smooth and let $A\rightarrow B'_0$ be the Albanese fibration of $f'$.  There is a rational map
\begin{eqnarray}\label{eqpsi} \psi:  \widetilde{X}\dashrightarrow A,
\end{eqnarray}
which is constructed as follows: we define
 $\psi$ as  the composition of the relative Abel or Albanese map up to isogeny
\begin{eqnarray}\label{eqabelmap}  {\rm alb}: \widetilde{X}  \dashrightarrow  J( \widetilde{X}/B),\end{eqnarray}
that we used previously and which to $c\in   \widetilde{X}_b$ associates ${\rm alb}_{\widetilde{X}_b}((2g-2)c-K_{\widetilde{X}_b})$,
and the natural rational map
\begin{eqnarray}\label{eqpsiab}\psi_{\rm ab}: J( \widetilde{X}/B)    \dashrightarrow A,
\end{eqnarray}
 inducing over a general  $b\in B$ the morphism

\begin{eqnarray}\label{eqsansmon} \psi_{{\rm ab},b}:J(\widetilde{X} _b)= {\rm Alb}(\widetilde{X}_b) \rightarrow {\rm Alb}(\widetilde{X} _{b'}),\,\,b'=m'(b)
\end{eqnarray}
of abelian varieties.
\begin{rema} \label{rematrestardive} {\rm The rational map $\psi$ might be different from any relative Albanese map for $f'$ constructed using a multisection of $f'$. More precisely, it may differ from it by translation by a rational section of $A$ over $B$, that is a rational map $B\dashrightarrow A$ over $B'$.}
\end{rema}
We have
\begin{lemm} \label{lemmecrucialpourfibgen}  The assumptions being as in Lemma \ref{levgeemen}, the  image $Y:={\rm Im}\,\psi\subset A$ has dimension ${\rm dim}\,B'+1$ and there is a  holomorphic  $2$-form $\eta$ on any  smooth projective birational model of $Y$, whose pull-back to $X$ under $\psi_Y:X\dashrightarrow Y$ is a nonzero multiple of $\sigma_X$.
\end{lemm}

\begin{proof} We first claim  that   for
general $b\in B$, with $m'(b)=b'$, the morphism of abelian varieties (\ref{eqsansmon}) is an isogeny on its image.
By  Lemma \ref{levgeemen}, the general fibers of $f':\widetilde{X}\rightarrow B'$ are isotropic for $\tau^*\sigma_X$, hence there is a  morphism
$$\sigma'_{b'}: T_{B',b'}\rightarrow H^0(\widetilde{X}_{b'},\Omega_{\widetilde{X}_{b'}})$$
of contraction with $\tau^*\sigma_X$.
By Remark
 \ref{remaquivaservir}, the morphism  (which is surjective by assumption  in Case  (i))
$$\sigma_{b}: T_{B,b}\rightarrow H^0(\widetilde{X}_{b},  \Omega_{\widetilde{X}_{b}}),$$
factors through an isomorphism
\begin{eqnarray}\label{eqoverlinesigmab} \overline{\sigma}_b: T_{B',b'}\rightarrow  H^0(\widetilde{X}_{b},  \Omega_{\widetilde{X}_{b}}).\end{eqnarray}
It is immediate to check that the following diagram is commutative
\begin{eqnarray}\label{eqnomdediagrammlenom} \begin{matrix}&  T_{B',b'}&\stackrel{\sigma'_{b'}}{\rightarrow} &  H^0(\widetilde{X}_{b'},\Omega_{\widetilde{X}_{b'}})\\
&\parallel &&  \psi_{\rm ab,b}^*\downarrow\\
&T_{B',b'}&\stackrel{ \overline{\sigma}_b}{\rightarrow}&  H^0(\widetilde{X}_{b},  \Omega_{\widetilde{X}_{b}}).
\end{matrix}
\end{eqnarray}
This implies that $ \psi_{\rm ab,b}^*$ is surjective, thus proving the claim.
It follows from the claim that the image ${\rm Im}\,\psi_{ab}$ is a family of abelian varieties  $J'\rightarrow B'$ over $B'$ which descends up to isogeny the family
$J(\widetilde{X}/B)\rightarrow B$. The image of the curve $\widetilde{X}_b$ in $J'_{b'}$ via $\psi$ obviously does not depend on the point $b$ in the fiber ${m'}^{-1}(b')\subset B$, since by construction of $\psi$, this is up to isogeny the curve $\widetilde{X}_b$ canonically embedded via  the Abel map (\ref{eqabelmap}). This  proves that ${\rm dim}\,Y={\rm dim}\,B'+1$.
It remains to construct a nonzero holomorphic   $2$-form  $\eta$ on $Y_{\rm reg}$ satisfying the desired property. Recall that $Y\subset A:={\rm Alb}(\widetilde{X}/B')$. Using a multisection of $\widetilde{X}\dashrightarrow B'$, we get a relative Albanese map (or rather a multiple depending on the degree of the multisection)
$a_{\widetilde{X}/B'}:\widetilde{X}\dashrightarrow A$. Furthermore,  by Lemma \ref{levgeemen}, using  \cite{bai} or \cite{voisintriangle},   the relative Albanese variety $A$ admits  a holomorphic $2$-form $\sigma_A$ (which extends to a smooth projective compactification of $A$) with the property that
\begin{eqnarray} \label{eqtardivepourpullback}  a_{\widetilde{X}/B'}^*\sigma_A=\sigma_{\widetilde{X}}.
\end{eqnarray}
Let $\eta:= \sigma_{A\mid Y_{\rm reg}}$. This $(2,0)$-form clearly extends to a smooth projective compactification of $Y_{\rm reg}$. It remains to prove that the pull-back  $\psi_Y^*\eta$, which by definition of $\eta$ equals  $\psi^*\sigma_A$, is nonzero on $\widetilde{X}$. This follows in fact directly from (\ref{eqtardivepourpullback}), by observing that the $(2,0)$-forms
$$\psi^*\sigma_A,\,\, a_{\widetilde{X}/B'}^*\sigma_A$$
on $\widetilde{X}$ differ by a holomorphic $(2,0)$-form on $B$, (see Remark \ref{rematrestardive}) and $B$ has no nonzero holomorphic $2$-form.
\end{proof}
Lemma \ref{lemmecrucialpourfibgen} provides us with a contradiction since ${\rm dim}\,Y={\rm dim}\,B'+1=g+1\leq n+2 <2n$ because $n\geq 3$ and thus the pull-back of $\eta$ to $\widetilde{X}$  provides a  holomorphic $2$-form on $\widetilde{X}$ which is everywhere degenerate. This case is thus excluded.

 \vspace{0.5cm}

 {\it Case  $g=n+1$ and  $\sigma_b$ has for image a hyperplane in $H^0(\widetilde{X}_b,\Omega_{\widetilde{X}_b})$.} We use the same notation as before, that is $K_b\subset T_{B,b}$ is the kernel of the contraction map $\sigma_b:T_{B,b}\rightarrow H^0(\widetilde{X}_b,\Omega_{\widetilde{X}_b})$. In this case, we first have the following variant of  Lemma \ref{leKb}:
 \begin{lemm}\label{leKbprime} At  a general point $b\in B$, the rank of the map
 $$\rho_b:K_b\rightarrow H^1(\widetilde{X}_{b},  T_{\widetilde{X}_{b}})$$
 is at most $1$.
 \end{lemm}
 \begin{proof} By the same arguments as in the proof of Lemma \ref{leKb}, we find that
 $\rho_b(K_b)$ is orthogonal with respect to Serre duality to $H^0(\widetilde{X}_{b},  K_{\widetilde{X}_{b}})\cdot {\rm Im}\,\sigma_b\subset H^0(\widetilde{X}_{b},  2K_{\widetilde{X}_{b}})$.
 As we know by Proposition  \ref{theofibgen} that the general fiber  $\widetilde{X}_{b}$ is not hyperelliptic, and by assumption ${\rm Im}\,\sigma_b\subset H^0(\widetilde{X}_{b},  K_{\widetilde{X}_{b}})$ is a hyperplane, $H^0(\widetilde{X}_{b},  K_{\widetilde{X}_{b}})\cdot {\rm Im}\,\sigma_b$ has codimension at most $1$ in $ H^0(\widetilde{X}_{b},  2K_{\widetilde{X}_{b}})$, which proves the lemma.
 \end{proof}
 As ${\rm rank} \, \sigma_b=n$, we have ${\rm dim}\,K_b=n-1\geq 2$, and it follows from Lemma \ref{leKbprime} that
 ${\rm Ker}\,\rho_b\not=0$, that is, the moduli map has positive dimensional general fiber.
 The rest of the proof works as in the previous case, except that the morphism of abelian varieties $\psi_{{\rm ab},b}$ of (\ref{eqsansmon} ) can now have a $1$-dimensional kernel, so that only a $g-1$-dimensional quotient of the Jacobian fibration descends to $B'$.
\end{proof}
We conclude this section with the proof of Corollary \ref{corotardif}.
\begin{proof}[Proof of Corollary \ref{corotardif}] Let $X$ be a very general hyper-K\"{a}hler $2n$-fold with $n\geq 3$ and  $b_2(X)_{\rm tr}\geq 9$. By Theorem \ref{theofibgen}, one has ${\rm fibgen}(X)\geq 5$ and by Theorem \ref{theomir}, one has ${\rm cohirr}(X)\geq 4$, hence a fortiori ${\rm RCirr}(X)\geq 4$. It thus follows from Proposition \ref{theocomp} that ${\rm fibgon}(X)\geq 4$.
\end{proof}
\section{Measure of irrationality \label{secmir} }
We do not know if (maybe under some assumptions on $b_{2,{\rm tr}}(X)$) the cohomological irrationality of a hyper-K\"{a}hler fourfold $X$ is at least $4$, which would greatly simplify the proof of Theorem \ref{propLMnef3}, but we can prove a weaker statement that will be used in the next section.
  \begin{prop}\label{prodeg3} Let $X$ be a hyper-K\"ahler fourfold such that  any big divisor on $X$ is ample. Then there exists no quasi-finite morphism
$f: X\rightarrow Y$ of degree $3$, where $Y$ is projective, normal and $-K_{Y_{\rm reg}}$ is  the restriction to the regular locus $Y_{\rm reg}$ of  a big line bundle on $Y$.
\end{prop}
Here,  by a ``big line bundle on $Y$'', we mean  the sum of an ample line bundle and an effective divisor. Our assumptions on $K_{Y_{\rm reg}}$ thus mean that  there exists an ample line bundle $\mathcal{L}$ on $Y$, and an effective divisor $E$ in $Y$, such that
\begin{eqnarray}\label{eqpourbigK} K_{Y_{\rm reg}}=\mathcal{L}^{-1}(-E)_{\mid Y_{\rm reg}}.\end{eqnarray}
\begin{proof} We first observe that, under our assumptions, $h^0(\widetilde{Y},K_{\widetilde{Y}})=0$ for any desingularization $\widetilde{Y}$ of $Y$. Indeed, using (\ref{eqpourbigK}), we get that
$$H^0(Y_{\rm reg},K_{Y_{\rm reg}})\subset H^0(Y_{\rm reg},\mathcal{L}^{-1}_{\mid Y_{\rm reg}})$$
and the right hand side is zero since $Y$ is normal and $\mathcal{L}$ is ample.
 This fact also implies that  $h^{2,0}(\widetilde{Y})=0$, since otherwise the $2$-form on $X$ would be pulled-back from $Y$, hence also its $(4,0)$-form, while we know that $h^{4,0}(\widetilde{Y})=0$.

The ramification divisor $R$ of $f$, which is well-defined on $f^{-1}(Y_{\rm reg})$, belongs to the linear system $|f^*(-K_Y)|$, hence is big on $f^{-1}(Y_{\rm reg})$.  There is a second effective divisor $R'$  in $f^{-1}(Y_{\rm reg})\subset X$, namely $f^{-1}(f(R))-2R$. The divisor $R'$ is not empty since its image in $Y_{\rm reg}$ is equal to $f(R)$.
We now prove
\begin{lemm} \label{lenomdelemmesansnom} The locus defined as the  intersection
\begin{eqnarray}\label{eqinterRRprime} S:= R\cap R'
\end{eqnarray}
in $X^0:=X - f^{-1}(Y_{\rm sing})$ is isotropic for the $2$-form $\sigma_X$.
\end{lemm}
\begin{proof} We observe that, due to the fact that the map $f$ is quasi-finite (hence finite over the smooth locus of $Y$), the locus (\ref{eqinterRRprime})
consists of points $x\in X$ such that the length of the fiber $f^{-1}(f(x))$ at $x$ is at least $3$, hence equal to  $3$ since the degree of $f$ is 3. For all these points $x\in X^0$, the class $3x\in{\rm CH}_0(X)$ is thus the inverse image of a $0$-cycle of $Y$. It follows from Mumford's theorem \cite{mumford} that the restriction
of $\sigma_X$ to any  desingularization of $S$ is the pull-back of a $2$-form defined on $Y_{\rm reg}$, and in fact on a desingularization $\widetilde{Y}$ of $Y$. However, as mentioned above,  we have $h^{2,0}(\widetilde{Y})=0$.
\end{proof}
In order to finish the proof, we have to see what happens along the singular locus $Y_{\rm sing}$ of $Y$.
\begin{lemm} \label{le2pour34} Any $2$-dimensional component of $f^{-1}(Y_{\rm sing})$ is also Lagrangian for $\sigma_X$.
\end{lemm}
\begin{proof} Let $\Sigma_2$ be the union of the $2$-dimensional components of $\Sigma:={\rm Sing}\,Y$ and let $y\in \Sigma_2$ be a general point. We claim that $f^{-1}(y)$ consists of a single point.  By flattening, after blowing-up $Y$ to a smooth variety $\widetilde{Y}$, the exceptional fiber of $\tau:\widetilde{Y}\rightarrow Y$ has connected fiber over $y$, because $Y$ is normal, and it parameterizes schemes $z$ of finite length with support the fiber $f^{-1}(y)$. The local multiplicities of $z$ at any of its points $x\in f^{-1}(\{y\})$ cannot be $1$ as otherwise the local degree of $f$ near the point $x$ would be $1$ and, by normality, $f$ would be a local isomorphism, contradicting the fact that $Y$ is singular at $y$. This implies that $f^{-1}(\{y\})$ contains at most one point since the sum of the local degrees over $Y_{\rm reg}$ is $3$.
The argument above shows that points of $\widetilde{Y}$ over $y\in Y_{\rm sing}$  parameterize  subschemes of length $3$ supported at a single point $x\in X$ over $y$. It thus follows again by Mumford's theorem \cite{mumford} that the restriction of $\sigma_X$ to $f^{-1}(\Sigma_2)$ is the restriction of a $2$-form on $\widetilde{Y}$, hence $0$ by the  argument already used.
\end{proof}
We now consider  the Zariski closures $\overline{R}$  of $R$ and $\overline{R'}$  of $R'$.
\begin{coro} The intersection $\overline{R}\cap \overline{R'}$ is isotropic for $\sigma_X$. In particular, it has dimension $2$ since there is no  divisor in $X$ which is isotropic for $\sigma_X$.
\end{coro}
Indeed, this is true away from $f^{-1}(Y_{\rm sing})$ by Lemma \ref{lenomdelemmesansnom}  and over $Y_{\rm sing}$ by Lemma \ref{le2pour34}.

The contradiction now comes from the fact that $\overline{R'}$ is a non-empty divisor in $X$, so that the restriction $\overline{\sigma}$ of   $\sigma_{X}$ to $\overline{R'}$, or rather its pull-back to a desingularization $\tau:\widetilde{\overline{R'}}\rightarrow \overline{R'}$ of $\overline{R'}$, is nonzero. As the ramification  divisor $\overline{R}$ is a big divisor on $X$ since it is  linearly equivalent to $f^*(-K_Y)$ over $Y_{\rm reg}$, it is an ample divisor by our assumptions, hence  its pull-back ${\tau'}^*{\overline{R}}$ to $\widetilde{\overline{R'}}$ is big, where $\tau':\widetilde{\overline{R'}}\rightarrow X$ is the composition of $\tau$ and the inclusion map $\overline{R'}\hookrightarrow X$. This contradicts the fact that the surface
$\overline{R}\cap \overline{R'}\subset \overline{R'}$, hence also its inverse image in $\widetilde{\overline{R'}}$,  is isotropic for the $2$-form $\overline{\sigma}$ on $\widetilde{\overline{R'}}$.
\end{proof}

\section{Rational maps from  hyper-K\"{a}hler fourfolds: a variant of a theorem of O'Grady}
In the paper \cite{ogradyccm}, O'Grady proves the following result.
\begin{theo}\label{theoogadry} Let $X$ be a hyper-K\"{a}hler fourfold which is numerically equivalent to ${\rm K3}^{[2]}$. Assume $\rho(X)=1$ and ${\rm Pic}(X)$ is generated by one positive line bundle $H$ with $q_X(H)=2$, or equivalently, $H^4=12$. Then the rational map
$$\phi_H:X\dashrightarrow\mathbb{P}^5$$
is either birational to  a hypersurface of degree $12\geq d\geq 6$, or of degree $2$ over a hypersurface of degree $ 6$ whose desingularization has $p_g\not=0$.
\end{theo}
Here, ``numerically equivalent" means that the lattice $(H^2(X,\mathbb{Z}),q_X)$ is isomorphic to the corresponding lattice for ${\rm K3}^{[2]}$. As explained in {\it loc. cit.}, Theorem \ref{theoogadry} is equivalent to exclude the possibilities where the image of $\phi_H$ is of dimension $<4$ or a hypersurface of degree $<6$. In these two cases, the image would be  rationally connected by \cite{lin}.

In this section, we are going to extend Theorem \ref{theoogadry} to the situation studied in \cite{DHMV}. The hyper-K\"ahler fourfold $X$ is only supposed to be very general with $\rho(X)=2$ and to admit two line bundles $L$ and $M$ satisfying
\begin{eqnarray}
\label{eqpourinfibdr} L^4=M^4=0,\,\,L^2M^2=2,
\end{eqnarray}
 which gives   $(L+M)^4=12$ since this implies by \cite{bogo}
 \begin{eqnarray}
\label{eqpourinfibdrsupinfo} L^3M=LM^3=0.
\end{eqnarray}
It is proved in \cite[Theorem 1.7]{DHMV} that such an $X$ has $b_2(X)=23$ and the same Chern numbers and Fujiki constant as ${\rm K3}^{[2]}$, and that the Riemann-Roch polynomial $\chi(X,kL+k'M)$ coincides with the similarly defined  polynomial on ${\rm K3}^{[2]}$ equiped with line bundles $L$, $M$ satisfying (\ref{eqpourinfibdr}); however we do not know  a priori  that $X$ is numerically equivalent  to ${\rm K3}^{[2]}$. The following result is in fact needed in order to prove that $X$ as above is deformation equivalent to ${\rm K3}^{[2]}$
 so that, a posteriori, $X$ is numerically equivalent  to ${\rm K3}^{[2]}$ (see \cite[Theorem 1.5]{DHMV}).
\begin{theo} \label{propLMnef3} Assume that $X,\,L,\,M$ are as above, with $L,\,M$ nef and  $X$ very general with $\rho(X)=2$, and that  $h^0(X,L)=0$.  Then $\phi_{L+M}:X\dashrightarrow\mathbb{P}^5$ does not have rationally connected  image.
\end{theo}
\begin{rema}\label{remapourirredpsefnef} {\rm The conditions $L,\,M$ nef and $h^0(X,L)=0$ imply  that  no divisor in $|L+M|$ is reducible. Indeed, if $L$ and $M$ are nef, any effective divisor $D$ on $X$ satisfies $q(L,D)\geq 0$ and $q(M,D)\geq 0$ so that $D$ is a combination with integral nonnegative coefficients of $L$ and $M$.}
\end{rema}

Note that \cite[Proposition 6.3]{DHMV} proves that $\phi_{L+M}:X\dashrightarrow\mathbb{P}^5$  has rationally connected  image, so that, in fact, an $X$ as above, with $L$ and $M$ nef satisfying (\ref{eqpourinfibdr}) and $h^0(X,L)=0$ does not exist.

The proof of  Theorem  \ref{propLMnef3} will be done in several steps. Although the statement is very similar to Theorem \ref{theoogadry},   we cannot use the strategy of O'Grady, who proves first  that any surface which is the complete intersection of two members of $|L+M|$ is reduced and irreducible, a statement that is  a priori not true in our situation.  Nevertheless, using the fact that $(L+M)^4=12$, and under the assumption   that no divisor in $|L+M|$ is reducible,   a number of his arguments go through in our situation where $\rho(X)=2$ and $L,\,M$ are nef.

The following lemma will be very much used in the proof. We denote $l=c_1(L)\in{\rm Hdg}^2(X,\mathbb{Z}),\,m=c_1(M)\in{\rm Hdg}^2(X,\mathbb{Z})$.
\begin{lemm}\label{nolmsur2crucial} Assume  $X$ is as above, very general with $\rho(X)=2$. Then there is no surface $\Sigma\subset X$, such that the class  $(l+m)^2-3[\Sigma]\in {\rm Hdg}^4(X,\mathbb{Z})$ is pseudoeffective.
\end{lemm}
\begin{proof} We argue as in the proof of \cite[Claim 6.2]{DHMV}. Any integral cohomology class
$\eta\in H^4(X,\mathbb{Z})$ has an associated matrix
\begin{eqnarray}\label{eqmatrice} M_\eta=
\begin{pmatrix}
a & b \\
b & c
\end{pmatrix},
\end{eqnarray}
with $a=\langle \eta,l^2\rangle_X$, $b= \langle \eta,ml\rangle_X$, $c=\langle \eta,m^2\rangle_X$.
If $\eta$ is the class of a surface in $X$, this matrix is nonzero since $L+M$ is ample and  has nonnegative coefficients since $L$ and $M$ are nef.
We follow some computations and arguments  of  \cite{ogradyccm}, which we can do as we are in a very similar numerical situation, namely our $X$ has by \cite[Theorem 1.7]{DHMV} the same Chern numbers, Betti numbers and Fujiki constant as ${\rm Hilb}^2(K3)$. As $b_2(X)=23$, one has an isomorphism given by cup-product (see \cite{bogo}, \cite{guan})
$${\rm Sym}^2H^2(X,\mathbb{Q})\cong H^4(X,\mathbb{Q}),$$ which induces a decomposition
\begin{eqnarray}\label{eqdecompH4}H^4(X,\mathbb{Q})={\rm Sym}^2H^2(X,\mathbb{Q})_{\rm tr}\oplus H^2(X,\mathbb{Q})_{\rm tr}\otimes {\rm NS}(X)_{\mathbb{Q}}\oplus {\rm Sym}^2{\rm NS}(X)_{\mathbb{Q}}.
\end{eqnarray}
As $X$ is very general, the Mumford-Tate group of the Hodge structure on $H^2(X,\mathbb{Q})_{\rm tr}$ is the orthogonal group of the Beauville-Bogomolov form $q_X$, so that the only Hodge classes in ${\rm Sym}^2H^2(X,\mathbb{Q})_{\rm tr}\subset H^4(X,\mathbb{Q})$ are multiples of the class $c$ inducing the Beauville-Bogomolov form.
By (\ref{eqpourinfibdr}) and (\ref{eqpourinfibdrsupinfo}), the classes $l^2$ and $m^2$ satisfy
\begin{eqnarray}\label{eqdematl2m2}M_{l^2}=\begin{pmatrix}
0 & 0 \\
0 & 2
\end{pmatrix},\,\,M_{m^2}=\begin{pmatrix}
2 & 0 \\
0 & 0
\end{pmatrix},\end{eqnarray}
while the integral Hodge classes $lm$ and $c_2(X)$ satisfy
\begin{eqnarray}\label{eqdematlmc2}M_{lm}=\begin{pmatrix}
0 & 2 \\
2 & 0
\end{pmatrix},\,\,M_{c_2(X)}=\begin{pmatrix}
0 & \lambda \\
\lambda & 0
\end{pmatrix},\end{eqnarray}
with $\lambda=30$ as for a hyper-K\"{a}hler fourfold of $K3^{[2]}$ deformation type.
It is indeed a general fact that the Beauville-Bogomolov form for hyper-K\"{a}hler fourfolds is a nonzero multiple of the quadratic form $q_{c_2(X)}(\alpha,\beta)=\langle \alpha\beta,c_2(X)\rangle_X$ on $H^2(X,\mathbb{Q})$. The computation of the coefficient $\lambda$ is as in the case of $K3^{[2]}$ since it is determined by the Riemann-Roch polynomial and the Fujiki constant. It follows from (\ref{eqdecompH4}) that the space of rational  Hodge classes on $X$ is generated by ${\rm Sym}^2{\rm NS}(X)_{\mathbb{Q}}$ and $c$, and the kernel of the map $\eta\rightarrow M_\eta$ on ${\rm Hdg}^4(X,\mathbb{Q})$ is of  rank $1$, generated by $c_2(X)-15ml$.

 Let $f=[\Sigma]$ and $e=(l+m)^2-3f\in H^4(X,\mathbb{Z})$ be the two considered pseudoeffective classes.
The corresponding matrices $M_e$ and $M_f$ thus satisfy
\begin{eqnarray}\label{eqdecompmat}
3M_f+M_e=\begin{pmatrix}
2 & 4 \\
4 & 2
\end{pmatrix}\end{eqnarray}
and as both matrices are nonzero, with integral nonnegative coefficients, we must have
\begin{eqnarray}\label{eqMfMe} M_f=\begin{pmatrix}
0 & 1 \\
1 & 0
\end{pmatrix},\,\, M_e=\begin{pmatrix}
2 & 1 \\
1 & 2
\end{pmatrix}.
\end{eqnarray}
Note that \begin{eqnarray}\label{eqquelquesM}\begin{pmatrix}
0 & 1 \\
1 & 0
\end{pmatrix}=M_{\frac{1}{2}ml},\,\,\begin{pmatrix}
2 & 1 \\
1 & 2
\end{pmatrix}=M_{l^2+m^2+\frac{1}{2}ml}.\end{eqnarray}

It follows from (\ref{eqMfMe}) and (\ref{eqquelquesM}) that for some coefficient $\eta\in\mathbb{Q}$
we have
\begin{eqnarray}\label{eqvaleurdeeetf} f=\frac{1}{2}ml+\eta(c_2(X)-15ml),\,\,e=l^2+m^2+\frac{1}{2}ml-3\eta(c_2(X)-15ml).
\end{eqnarray}
We now compute the self-intersection of these integral cohomology classes and conclude that
$$f^2=\frac{1}{2}+\eta^2(c_2(X)-15ml)^2=\frac{1}{2}+378\eta^2.$$
 We thus conclude that $2\cdot 378\eta^2$ is an integer, and as  $378=27\cdot2\cdot7$, it follows  that $6\eta$ is an integer.  From the first equation in (\ref{eqvaleurdeeetf}), with $f$ effective, we now conclude   that $\eta <0$ since otherwise $\eta\geq \frac{1}{6}$ and $\frac{1}{2}-15\eta<0$, so
 $$\eta c_2(X)=f+(15\eta-\frac{1}{2})lm$$
 with all coefficients positive and $f$ effective. This is equation (32) in  \cite[Proof of Claim 6.2]{DHMV} which is  proved there   to be impossible.

 From the second equation in (\ref{eqvaleurdeeetf}), we now deduce that
 \begin{eqnarray}\label{eqpoureavecc2} -3\eta c_2(X)=e-l^2-m^2 +(-45\eta-\frac{1}{2})ml.\end{eqnarray}
 We claim that this implies $\eta\geq-\frac{1}{18}$. This is proved by integrating against both terms of  (\ref{eqpoureavecc2}) a class $\alpha^2$, where $\alpha\in H^{1,1}(X)_{\mathbb{R}}$ is  in the boundary of the K\"{a}hler cone and satisfies   $q(\alpha)=0$. We get
 \begin{eqnarray}
 \label{eqquationintege} 0=\langle e,\alpha^2\rangle_X-\langle l^2,\alpha^2\rangle_X-\langle m^2,\alpha^2\rangle_X+(-45\eta-\frac{1}{2})\langle lm,\alpha^2\rangle_X.
 \end{eqnarray}
 Using the Fujiki relations (with Fujiki constant equal to $3$), we have
 $$\langle \beta\gamma,\alpha^2\rangle_X=2q_X(\alpha,\gamma)q_X(\alpha,\beta)$$
 for any $\alpha,\,\beta,\,\gamma\in H^2(X,\mathbb{C})$ such that $q_X(\alpha)=0$.
Thus (\ref{eqquationintege}) gives
$$0=\langle e,\alpha^2\rangle_X-2q_X(l,\alpha)^2-2q_X(m,\alpha)^2+2(-45\eta-\frac{1}{2})q_X(l,\alpha)q_X(m,\alpha)$$
and, as $e$ is pseudoeffective, $\langle e,\alpha^2\rangle_X\geq 0$ when $\alpha$ is in the boundary of the K\"{a}hler cone, which by \cite[Proposition 3.2]{huybrechts}, is satisfied once $q_X(l,\alpha)\geq0,\,q_X(m,\alpha)\geq0$. In conclusion, we proved that
$$q_X(l,\alpha)^2+q_X(m,\alpha)^2+(45\eta+\frac{1}{2})q_X(l,\alpha)q_X(m,\alpha)\geq 0
$$
once $q_X(l,\alpha)\geq0,\,q_X(m,\alpha)\geq0$.
It follows that
$45\eta+\frac{1}{2}\geq -2$, which proves the claim.

As we know that $6\eta$ is an integer and $\eta<0$, the claim gives a contradiction proving the lemma.
\end{proof}

The proof of Theorem \ref{propLMnef3} will be obtained by a case by case study. Assuming $\phi_{L+M}$ has rationally connected image, we have, by adapting arguments of  \cite{ogradyccm}, the following three possibilities (the case where the image is a curve being  directly excluded by the fact that no divisor in $|L+M|$ is reducible).
\begin{enumerate}
\item \label{itemsurf}  $Y=\phi_{L+M}(X)\subset \mathbb{P}^5 $ is a surface of degree $d\geq 4$.
\item \label{item3fold}  $Y=\phi_{L+M}(X)$ is a $3$-fold  of degree $3\leq d\leq 6$. In the case of degree $d=6$, the indeterminacy locus of $\phi_{L+M}$ has dimension $0$.
\item \label{item4fold} $Y=\phi_{L+M}(X)$ is a $4$-fold  of degree $2\leq d\leq 4$ and the degree of $\phi_{L+M}:X\dashrightarrow Y$ is at least $3$.
\end{enumerate}
The  bound on the degree $d$ in (\ref{itemsurf}) follows from the fact that the image $Y$ is linearly nondegenerate in $\mathbb{P}^5$.

The  bound on the degree $d$ in (\ref{item3fold}) follows from the fact that the image $Y$ is linearly nondegenerate in $\mathbb{P}^5$ and that the general fiber is a curve $F$ such that $d[F]+e=(m+l)^3$ in ${\rm Hdg}^{6}(X,\mathbb{Z})$ for some pseudoeffective class $e$ (we use the ampleness of $L+M$ here). Indeed,  as we assumed $\rho(X)=2$, the group ${\rm Hdg}^{6}(X,\mathbb{Q})$ is generated by $l^2m$ and $m^2l$. An integral pseudoeffective curve  class  in $X$ can thus be written as $\alpha l^2m+\beta m^2l$. By intersecting with $l$ and $m$, we find that $2\alpha$ and $2\beta$ are integers, and furthermore, they are nonnegative since $L$ and $M$ are nef. Applying this argument to $[F]$ and $e$, and using $(m+l)^3=3m^2l+3ml^2$, the equality $d[F]+e=(m+l)^3$, with $[F]$ and $e$ pseudoeffective implies that $d\leq 6$.

The bound on the degree $d$ in (\ref{item4fold}) follows from ampleness of $L+M$ and the fact that $(l+m)^4=12$. Furthermore, as in \cite{ogradyccm} (see also Lemma \ref{proinvol}), one uses the fact that the degree of $X$ over $Y$ is at least $3$ since $p_g(\widetilde{Y})=0$. Here and in the sequel, we  denote by $\widetilde{Y}$ a desingularization of $Y$ and $\tilde{\phi}:\widetilde{X}\rightarrow \widetilde{Y}$ a  desingularization of $\phi: X\dashrightarrow \widetilde{Y}$.

We thus have to exclude each of these possibilities. Let us start by excluding a few easy cases.
\begin{lemm}\label{leeasyA} The image  $Y\subset \mathbb{P}^5 $  of $\phi_{L+M}$ is not a surface of degree $d\geq 4$.
\end{lemm}
\begin{proof} Otherwise, the general fiber $F$ is a surface in $X$ such that
$(l+m)^2-d[F]=e$, where $e$ is the class of a surface (which is a union of irreducible  components of the base-locus of $|L+M|$), and this  is excluded by Lemma \ref{nolmsur2crucial}.
\end{proof}
\begin{lemm}\label{leeasyB} The image $Y\subset \mathbb{P}^5 $  of $\phi_{L+M}$ is not a threefold  of degree $3$.
\end{lemm}
\begin{proof} By \cite{eisenbudharris}, a  linearly normal $3$-fold $Y$ of degree $3$ in $\mathbb{P}^5$ is a cone over a rational normal scroll. Such a  $Y$ is   fibered into linear spaces over $\mathbb{P}^1$ and  has many reducible hyperplane sections, in the sense that it is swept-out by reducible hyperplane sections, with at least two mobile irreducible components. In that case,  $X$ would thus have, by taking pullback under $\phi_{L+M}$, reducible divisors in $|L+M|$, contradicting our assumption that $h^0(X,L)=0$ (see Remark \ref{remapourirredpsefnef}).
\end{proof}
\begin{lemm}\label{leeasyC} The image $Y\subset \mathbb{P}^5 $  of $\phi_{L+M}$ is not a fourfold  of degree $4$.
\end{lemm}
\begin{proof} By item \ref{item4fold} above, the rational map $\phi_{L+M}:X \dashrightarrow Y$ has degree $\geq 3$. As $(L+M)^4=12$ and $L+M$ is ample, the case where ${\rm dim}\,Y={\rm deg}\,Y=4$ is possible only if $\phi_{L+M}$ is a morphism of degree $3$ (see \cite[Corollary 4.7]{ogradyccm}). As $L+M$ is ample, the morphism  $\phi_{L+M}$ is quasifinite to its image and the same is true for the induced morphism $\phi_{L+M}:X\rightarrow Y_n$, where $Y_n$ is the normalization of $Y$. The big divisors are ample on $X$ since, by Remark \ref{remapourirredpsefnef}, the pseudoeffective cone of $X$ is generated by two nef line bundles, and  the regular locus of $Y_n$ has a big anticanonical bundle, hence  this would  contradict Proposition \ref{prodeg3}.
\end{proof}

\begin{lemm}\label{le3sur3dim4}  If the image $Y \subset \mathbb{P}^5$ of $\phi_{L+M }$ is a hypersurface of degree $3$, the degree of
$\phi_{L+M } : X \dashrightarrow Y$ is $3$.
\end{lemm}
\begin{proof} The rational map $\phi_{L+M }$ is of  degree $\geq 3$ by item \ref{item4fold} above, and it cannot be of degree $\geq 5$  since
$(L+M)^4 = 12 \geq {\rm  deg}\, Y {\rm deg }\,\phi_{L+M }$, because $L+M$ is ample (see \cite{ogradyccm}). So we have to exclude
the case where ${\rm deg}\, \phi_{L+M } = 4$ and ${\rm deg}\, Y =3$, where the equality
$(L +M)^4 = 12 = {\rm deg}\, Y {\rm deg}\, \phi_{L+M }$
holds, implying that $\phi_{L+M }$ is a morphism (see \cite[Corollary 4.7]{ogradyccm}). Let $C \subset  Y$ be a general plane section and
$C_X \subset X$ be its inverse image in $X$. We observe that $Y$ cannot be singular in codimension $1$,
as otherwise it has reducible hyperplane sections, hence, by taking the inverse images under
the morphism $\phi_{L+M }$, X has reducible members in $|L + M|$. It follows that  the curve $C$ is a smooth
elliptic curve. We use now the results proved in the course of the proof of Proposition 6.4 and in  Lemma 6.8 of \cite{DHMV}. They imply that,
under our assumptions on $X,\ L,\, M$, the rational map $ \phi_{2L+M\mid C_X}$ factors through
the degree $4$ rational map $\phi_{L+M\mid C_X} : C_X \rightarrow  C$. Note that the linear systems $|L+M|$ and $|2L+M|$ on
$X$ have no fixed components. Indeed, this is clear for the first as $|L+M|$ has no reducible divisors; for the second one, as we assumed $h^0(X,L)=0$, and we have   $h^0(X,2L+M)=10$, $h^0(X,L+M)=6$, the only fixed component could be in $|2L|$ and we would then have $h^0(X,M)=10$, or it could be in $|M|$ and we would then have $h^0(X,2L)=10$. Both possibilities are easily ruled out, using \cite[Lemma 5.1]{DHMV} and \cite{huyxu} (see  \cite[Section 5.1]{DHMV} for the complete argument).  As the curve  $C_X$ is mobile, it follows that the linear systems $|L+M|$ and   $|2L+M|$
have no base points along $C_X$, hence the factorization of the morphisms mentioned above shows that the
linear systems $H^0(X,L+M)_{\mid C_X }, \, H^0(X, 2L+M)_{\mid C_X} $  are pulled-back from linear systems on $C$.
A fortiori, we get that the line bundle $(2L + M)_{\mid C_X}$ is pulled-back from a line bundle on $C$, hence the
degree of $(2L + M)_{\mid C_X}$ is divisible by $4$. This contradicts the fact that
$${\rm deg}\, (2L + M)_{\mid C_X} = (2L +M)(L +M)^3 = 3(2L +M)(L^2M +ML^2) = 18,$$
which is obtained using the equalities $L^2M^2 = 2, L^3M = 0, LM^3 = 0$ of (\ref{eqpourinfibdr}) and (\ref{eqpourinfibdrsupinfo}).
\end{proof}
\begin{lemm}\label{leeasyD} The image $Y\subset \mathbb{P}^5 $  of $\phi_{L+M}$ is not a fourfold  of degree $2$.
\end{lemm}

\begin{proof} If $Y$ is a quadric,  the general plane section $C$ of $Y$, defined by a $3$-dimensional vector subspace $W_3\subset H^0(\mathbb{P}^5,\mathcal{O}_{\mathbb{P}^5}(1))=H^0(X,L+M)$  is a smooth conic, as otherwise $Y$ is singular in codimension $1$ hence is reducible. We thus have $C\cong \mathbb{P}^1$ and   denote by $\mathcal{O}_{\mathbb{P}^1}(1)$ the degree $1$ line bundle on $C$. We recall from \cite[Proof of proposition 6.4]{DHMV} that, under our assumptions on $X,\,L,\,M$, assuming that $Y$ is a fourfold, and  given a general plane section $C$ of $Y$, the  mobile part $X_C$ of $\phi_{L+M}^{-1}(C)$, or equivalently  the Zariski closure of the  locus in $X\setminus{{\rm BL}(L+M)}$ which is defined by $W_3$,  is an irreducible  curve with the following properties. We denote below $\phi_{L+M,C}:X_C\rightarrow C$ the restriction of $\phi_{L+M}$ to $X_C$.

\begin{enumerate} \item ${\rm dim}\,H^0(X,L+M)_{\mid X_C}=3$.
\item \label{item2pour2} ${\rm dim}\,H^0(X,2L+M)_{\mid X_C}=5$ or $4$, and in the second case, $\phi_{2L+M}(X_C)$ is a rational cubic curve in $\mathbb{P}^3$.

    \item ${\rm dim}\,H^0(X,3L+2M)_{\mid X_C}\leq 8$.
    \end{enumerate}
   (a)   If ${\rm dim}\,H^0(X,2L+M)_{\mid X_C}=5$, denoting by $W_5$ the space $H^0(X,2L+M)_{\mid X_C}$, and by $W_3\cong {\rm Sym}^2W_2$ the space $H^0(X,L+M)_{\mid X_C}$, with $W_2:=H^0(C, \mathcal{O}_{\mathbb{P}^1}(1))$,  we study the multiplication maps
     $$\mu: W_2\otimes W_5\rightarrow H^0(X_C,(2L+M)_{\mid X_C}\otimes \phi_{L+M,C}^*\mathcal{O}_{\mathbb{P}^1}(1))$$
     with image $W'$, and
     $$\mu': W_2\otimes W'\rightarrow H^0(X,3L+2M)_{\mid X_C}$$
     with ${\rm rank}\,\mu'\leq8$.
     We get by Hopf lemma applied to both multiplication maps that ${\rm dim}\,W'=6$ or ${\rm dim}\,W'=7$. In the first case, the equality in Hopf lemma is satisfied, and in the second case, the equality in the Hopf lemma is satisfied. In both cases, we conclude that \begin{eqnarray}\label{eqpourW5} W_5=\phi_{L+M}^*H^0(C,\mathcal{O}_{\mathbb{P}^1}(4))=
     \phi_{L+M}^*H^0(C,\mathcal{O}_C(2)).
     \end{eqnarray}
      It follows that the rational morphism
     $\phi_{2L+M}:X\dashrightarrow\mathbb{P}^9$ factors rationally  through $Y$. Furthermore, the linear system $|2L+M|$ has no fixed component, as we already explained in the previous proof. We also observe that the quadric $Y$   must be of rank at least $5$, otherwise it would have  many reducible hyperplane sections, and   $X$ would contain  reducible divisors in $|L+M|$. It follows that ${\rm Pic}(Y\setminus{\rm Sing}\,Y)=\mathbb{Z}\mathcal{O}_Y(1)$.
     These facts, together with the equality (\ref{eqpourW5}) imply that we have an equality of divisors in $X$
     \begin{eqnarray}\label{eqnarray} 2L+M=2(L+M)-E,
     \end{eqnarray}
     where $E$ is an effective  divisor in $X$ contracted by $\phi_{L+M}$. Thus $E$ belongs to $|M|$ and must be irreducible and contracted  by $\phi_{L+M}$ to  an irreducible subvariety $W$ of $Y$. Furthermore, this equality
      induces an equality of spaces of sections
       \begin{eqnarray}\label{eqnarraysections} H^0(X,2L+M)=H^0(Y,\mathcal{O}_Y(2)\otimes\mathcal{I}_W).
     \end{eqnarray}
      As $H^0(X,2L+M)$ is of dimension $10$ (see \cite{DHMV}), $W$ imposes at most  $11$ conditions to the quadrics. On the other hand, $W$ must generate linearly at least a $\mathbb{P}^4$. Otherwise, $W\subset \mathbb{P}^3$ and $Y$ is swept-out by linear sections containing $W$. Thus  there would be reducible divisors in $|L+M|$, namely inverse images of general hyperplane sections of $Y$ containing $W$, which contain $E$ and a mobile  component, contradicting our assumptions. Finally $W$ cannot be a curve.   Otherwise this curve would have degree at least $4$ and  the map $E\rightarrow W$ would have $2$-dimensional fibers of class $F$;  choosing $3$ general points on the curve $W$ and two general hyperplanes in $\mathbb{P}^5$ containing these $3$ points, we would conclude that $(l+m)^2-3F$ is effective in $X$, which is  excluded by  Lemma \ref{nolmsur2crucial}.
      An irreducible linearly nondegenerate surface $W$ in $\mathbb{P}^4$ or $\mathbb{P}^5$  imposes at least $12$ conditions to quadrics (as the only irreducible linearly nondegenerate surface in $\mathbb{P}^4$ contained in three quadrics is a cubic scroll, residual of a plane in the complete intersection of two quadrics, see \cite[p 50]{harrisnotes}), and this is a contradiction.

     \vspace{0.5cm}

     (b) The other case, where ${\rm dim}\,H^0(X,2L+M)_{\mid X_C}=4$, and  $\phi_{2L+M}(X_C)$ is a rational cubic curve in $\mathbb{P}^3$, is still easier. Indeed, we prove as above (see \cite{DHMV}) that the rational map $\phi_{2L+M}$ factors through $\phi_{L+M}$, and thus there is a linear system on $Y$ which is of degree $3$ on the plane sections $C$ of $Y$. This is impossible under our assumptions since as we argued above, the quadric $Y$  has rank at least $5$ hence $Y$ has  cyclic Picard group  generated by $\mathcal{O}_Y(1)$.
\end{proof}
\begin{lemm}\label{le63fold} The image $Y\subset \mathbb{P}^5 $  of $\phi_{L+M}$ is not a threefold of degree $6$.
\end{lemm}
\begin{proof}   First of all we prove
\begin{claim}\label{claimpourle63} Assume ${\rm Pic}\,X$ is generated by $L$ and $M$ with $L$ and $M$ nef isotropic, and the image $Y\subset \mathbb{P}^5 $  of $\phi_{L+M}$ is  a threefold of degree $d=4,\,5$ or $6$. Then the  general fiber $F$ of $\phi_{L+M}$ is  a  curve of class $\frac{1}{2}(L^2M+M^2L)$, and it  is of genus $3$ when $d=6$.
\end{claim}
As before, by ``general fiber of $\phi_{L+M}$'', we mean ``general fiber of a desingularization $\tilde{\phi}_{L+M}:\widetilde{X}\rightarrow Y$ of $\phi_{L+M}$.
\begin{proof}[Proof of Claim \ref{claimpourle63}]  Using the fact that $L+M$ is ample, and arguing as in \cite{ogradyccm}, the image $f$ of $F$ in the group of $1$-cycles of $X$ modulo numerical equivalence (or in $H_2(X,\mathbb{Z})$) satisfies
\begin{eqnarray}\label{equationpourfibre3} df=(l+m)^3-e=3(l^2m+m^2l)-e,\end{eqnarray}
where the class $e$ is the class of a pseudoeffective $1$-cycle.
Under our assumptions on $L,\,M$, the group of pseudoeffective $1$-cycles is contained in the cone generated over  $\mathbb{Q}$
by $l^2m$ and $lm^2$. Furthermore, the classes $\frac{1}{2}l^2m$ and $\frac{1}{2}lm^2$ are integral and any integral cohomology class in $\langle l^2m,m^2l\rangle_{\mathbb{Q}}$ is an integral combination of $\frac{1}{2}l^2m$ and $\frac{1}{2}lm^2$ as one sees by intersecting them with $L$ and $M$.
It now follows from (\ref{equationpourfibre3}) with $d\geq 4$ that one of the following possibilities holds
\begin{eqnarray}\label{equationpourfibre33} f=\frac{1}{2}(l^2m+m^2l),\,f=\frac{1}{2}l^2m,\,\,f=\frac{1}{2}m^2l.
\end{eqnarray}
Next we observe that the image of $F$ in $X$ is an irreducible component of the intersection of three members of $|L+M|$, and it follows by adjunction that
\begin{eqnarray}\label{eqtardiverevision}{\rm deg}\,K_F\leq 3(l+m) f.\end{eqnarray}
If $f=\frac{1}{2}l^2m$ or $f=\frac{1}{2}m^2l$, we get from (\ref{eqtardiverevision}) that ${\rm deg}\,K_F\leq 3$, that is, $F$ is  of genus $0,\,1$ or $2$, which contradicts Proposition   \ref{theofibgen}. Hence we conclude that $f=\frac{1}{2}(l^2m+m^2l)$, which proves the first statement. When $d=6$, we have $e=0$ by the inequality   (\ref{equationpourfibre3}) and it follows as in \cite{ogradyccm} that the base locus of $|L+M|$ consists of isolated points.  The equality in (\ref{eqtardiverevision}) would  imply that the holomorphic Euler-Poincar\'{e} characteristic $\chi(\widetilde{Z},\mathcal{O}_{\widetilde{Z}})$ equals the holomorphic Euler-Poincar\'{e} characteristic $\chi(Z,\mathcal{O}_Z)$, where
$Z$ is the complete intersection of three members of $|L+M|$ and $\widetilde{Z}$ is its normalization. This is not possible since the normalization map $\widetilde{Z}\rightarrow Z$ is nontrivial because $Z$ is connected and $\widetilde{Z}$ is not connected, hence we conclude that the inequality is strict in (\ref{eqtardiverevision}). Hence we get in this case ${\rm deg}\,K_F<6$ and so the genus of $F$ is at most $3$, hence equal to $3$ by Proposition   \ref{theofibgen}.
\end{proof}
  We now concentrate on the case of degree $d=6$. As noted above,  the indeterminacy locus of $\phi_{L+M}$ consists of isolated points.
We first prove
\begin{claim} \label{claimtrivalonF} There is a single indeterminacy point $x\in X$.
\end{claim}
\begin{proof} Let $\tau : \widetilde{X}\rightarrow X,\,\tilde{\phi}_{L+M}:\widetilde{X}\rightarrow Y$ be a resolution of indeterminacies of $\phi_{L+M}$. As we know that $\tau$ has rank at most $1$ over  the indeterminacy points $x_1,\ldots,\,x_N$, each irreducible component of the canonical divisor $K_{ \widetilde{X}}$ of $\widetilde{X} $, defined as the zero-locus of the form $\tau^*\sigma_X^4$, appears with multiplicity at least $3$. If $F\subset \widetilde{X}$ is a general fiber, we know by Claim \ref{claimpourle63} that $K_{ \widetilde{X}}\cdot F=4$, hence it follows that $F$ meets a single irreducible component of $K_{ \widetilde{X}}$. This implies that $N=1$, as  the image of $F$ in $X$ passes through all indeterminacy points $x_i$.
\end{proof}
We now examine the order of vanishing of sections of $L+M$ at $x$.
\begin{claim}\label{claimavecnomtardif123} (i) There is no section of $L+M$ vanishing at $x$ to order $3$ or more.

(ii)  There exists a section of $L+M$ whose zero set is nonsingular at $x$. The rank of the evaluation map $e_x: H^0(X,L+M)\rightarrow \Omega_{X,x}\otimes (L+M)$ is exactly $1$.

(iii) Let $V_x\subset T_{X,x}$ be the hyperplane defined by any  linear form in ${\rm Im}\,e_x$. Then the rank of the evaluation map $H^0(X,L+M)\rightarrow {\rm Sym}^2 V_x^*\otimes (L+M)$ is $5$.
\end{claim}
\begin{proof} (i) We have $\tau^*(L+M)\cdot F=2$ by Claim \ref{claimpourle63}. If a section of $L+M$ vanishes to order $\geq3$, it thus vanishes on all the curves $\tau(F)$, hence on $X$.

\vspace{0.5cm}

(ii) If all sections of $L+M$ vanish to order $\geq 2$ at $x$, the local intersection number at $x$ of $4$ sections of $L+M$ forming a regular sequence is at least $16$, contradicting the fact that $(L+M)^4=12$. Suppose now that there are two sections $s,\,s'$ of $L+M$ with independent differentials at $x$. Choosing them general, they define a smooth surface $S\subset X$ passing through $x$. This surface is swept out by curves $\tau(F)$ contained in it, hence it follows by the same argument as before that a nonzero section in  $H^0(X,L+M)_{\mid S}$ cannot vanish at order $\geq 3$ at $x$. There cannot be a  complete intersection of three sections of $L+M$ which is smooth at $x$, (since there are at least $6$ curves $F_i$ passing through $x$ in such complete intersection), hence any section in  $H^0(X,L+M)_{\mid S}$ vanishes to order $\geq2$  at $x$. The space $H^0(X,L+M)_{\mid S}$ is $4$-dimensional and the evaluation map
$$e_{x,S}:H^0(X,L+M)_{\mid S}\rightarrow {\rm Sym}^2\Omega_{S,x}\otimes (L+M)$$
has rank at most $3$. Hence $e_{x,S}$ has a non trivial kernel providing a section whose restriction to $S$ is nonzero and vanishes to order $3$ at $x$. This contradiction proves (ii).

\vspace{0.5cm}

(iii) The argument is the same as before, since, denoting by $X_s\subset X$ the zero locus of a general section $s$ of  $L+M$ (so that $V_x=T_{X_s,x}$), any element of $H^0(X,L+M)_{\mid X_s}$ has $0$ differential at $x$ but cannot vanish to order $\geq 3$ at $x$. The conclusion thus follows from the fact that $H^0(X,L+M)_{\mid X_s}$ has dimension $5$.
\end{proof}
A  contradiction arises as follows: the $5$-dimensional space of quadrics on $\mathbb{P}(V_x)$ given by Claim \ref{claimavecnomtardif123}(iii) either has no base point, or is the space of quadrics vanishing at a point $u\in \mathbb{P}(V_x)$. In both cases, if we take three general sections of $H^0(X,L+M)_{\mid X_s}$, they provide a rational map
$X_s\dashrightarrow \mathbb{P}^2$ that is undefined only at $x$, at which  the three sections vanish at order $2$.
Blowing-up $x$ in $X_s$, and denoting $E_{x,s}$ the exceptional divisor over $x$, we get sections
of $(L+M)_{\mid X_s}(-2E_{x,s})$. The restricted rational map $\phi_{L+M\mid E_{x,s}}:E_{x,s}\dashrightarrow \mathbb{P}^2$  is given by a general  linear system of quadrics vanishing at one point in the second case, or by a linear system of quadrics without base points in the first case. It is thus generically finite of degree $\leq 4$. This contradicts however the fact that it factors as the composition of  the dominant rational map
$$  E_{x,s} \dashrightarrow Y_s\rightarrow  \mathbb{P}^2$$
where $Y_s$ is the hyperplane section of $Y$ defined by $s$ and the second map is a general linear   projection, hence has degree $6$.
\end{proof}

Combining Lemmas \ref{leeasyA}, \ref{leeasyB}, \ref{leeasyC}, \ref{leeasyD} and  \ref{le63fold} we find that, in order to
prove Theorem \ref{propLMnef3}, we only have to prove the following Proposition \ref{procubquatre},  which eliminates the case where the image is a cubic hypersurface    and Proposition \ref{proquatrecinqtrois}, which excludes the cases where $Y$ is a $3$-fold of degree $4$ or $5$ in $\mathbb{P}^5$.

\begin{prop} \label{procubquatre} The image   $Y=\phi_{L+M} (X )\subset  \mathbb{P}^5$ cannot be
 a cubic hypersurface.
\end{prop}
We  establish  a few lemmas in order to  prove Proposition \ref{procubquatre}.  We first prove
\begin{lemm}\label{le1pourcubicdim4} If $Y$ is a cubic hypersurface, it cannot be singular in codimension $1$.
\end{lemm}
\begin{proof} If the singular locus of $Y$ has dimension $3$, it must be a $\mathbb{P}^3$ and, either $Y$ is a cone over a cubic surface, or the equation of $Y$ takes the form
\begin{eqnarray}\label{eqpourycas33} f_Y= x_0^2x_2+ x_0x_1x_3+ x_1^2 x_4, \end{eqnarray}
for an adequate choice of coordinates $x_i$, $x_0=x_1=0$ being the equations defining the $\mathbb{P}^3$ contained in ${\rm Sing}\,Y$. The first case is excluded as follows: If $Y$ is a cone over a cubic surface $S$, the linear projection $\pi:Y\dashrightarrow S$ from the vertex composes with $\phi_{L+M}$ to give a dominant rational map
$$\psi=\pi\circ \phi_{L+M}:X\dashrightarrow S$$
with general fiber $F_x$, $x\in S$.
For any general set $\{x_1,\,x_2,\,x_3\}$ of  three collinear points in $S$,
 the three surfaces $F_{x_i}$ are homologous in $X$ and satisfy
$[F_{x_1}]+[F_{x_2}]+[F_{x_3}]+e=(l+m)^2$ in $H^4(X,\mathbb{Z})$, where $e$ is the class of an effective surface in $X$, which contradicts Lemma \ref{nolmsur2crucial}. In the second case where $Y$ is defined by an equation $f_Y$ as in (\ref{eqpourycas33}),  $Y$ has  many reducible hyperplane sections. Indeed, in the above coordinates the hyperplane section
$\{x_2=0\}$ is the union of the two components $\{x_1=x_2=0\}\subset Y$, and $\{ x_0x_3+x_1x_4=x_2=0\}\subset Y$. Using the natural ${\rm SO}(3)$ (or ${\rm SL}(2)$) action on $Y$, it is easy to see that both components are mobile. Thus $X$ would have reducible members in $|L+M|$, which is excluded by assumption.
\end{proof}
By Lemma \ref{le1pourcubicdim4}, if $Y$  is a cubic hypersurface, the general plane sections $C:=P\cap Y$ are smooth elliptic plane curves.  We now choose a  desingularization $\widetilde{Y}$ of  $Y$, so that $Y_{\rm reg}\subset \widetilde{Y}$,  and  prove
\begin{lemm} \label{le2pourcubicdim4} If $Y$ is a cubic hypersurface, there exists a line bundle $\mathcal{L}$ on  $\widetilde{Y}$  such that
${\rm deg}\,\mathcal{L}_{\mid C}=5$ and the pull-back $\phi_{L+M}^*\mathcal{L}$ of $\mathcal{L}$ to $X$ satisfies
\begin{eqnarray}\label{eqpourpullbacl} 2L+M=\phi_{L+M}^*\mathcal{L}(-E),
\end{eqnarray}
for some effective divisor $E$ in  $X$ which is contracted by $\phi_{L+M}: X\dashrightarrow\widetilde{Y}$. Furthermore, the sections of $2L+M$ are pulled-back from sections of $\mathcal{L}$ on $\widetilde{Y}$. In particular
\begin{eqnarray}\label{eqpourpullbaclsec} h^0(Y_{\rm reg},\mathcal{L}_{\rm reg})\geq h^0(X,2L+M)=10
\end{eqnarray}
where $\mathcal{L}_{\rm reg}:=\mathcal{L}_{\mid Y_{\rm reg}}$.
\end{lemm}
\begin{proof} Denote  $D\subset X$  the curve $\phi_{L+M}^{-1}(C)$ (that is,  the mobile part of the closed algebraic subset defined by the three equations $\alpha,\beta,\gamma$ of $L+M$ on $X$ corresponding to the three sections of $\mathcal{O}_Y(1) $ defining $C$). We recall from \cite[Proof of proposition 6.4]{DHMV} that, under our assumptions,  the
linear systems
$$W_3:=H^0(X,L+M)_{\mid D},\,W_5:=H^0(X,2L+M)_{\mid D},\,W_8:=H^0(X,3L+2M)_{\mid D}$$
are of respective dimension $3,\,\geq 5,\,\leq 8$.  Then \cite[Lemma 6.8]{DHMV} proves, using the multiplication map
$$\mu: W_3\otimes W_5\rightarrow W_8$$
 that
these three linear systems are pulled back from linear systems $W'_3,\,W'_5,\,W'_8$ on the curve $C$.
By removing the base-points, we may assume  that the linear systems $W'_3$, $W'_5$ and $W'_8$ have no base-points on $C$. This defines a line bundle $\mathcal{L}_C$ on $C$ such that $W'_5\subset H^0(C,\mathcal{L}_C)$ and has no base-points.  Note that $W'_3$ gives the embedding of $C$ as a plane curve.
\begin{claim}\label{claimpourCetdeg8}  One has $ W'_5=H^0(C,\mathcal{L}_C)$; equivalently,  the line bundle $\mathcal{L}_C$ on $C$ has degree $5$.
\end{claim}
\begin{proof}  We have a base point free not necessarily complete  linear system $W'_5\subset H^0(C,\mathcal{L}_C)$ of dimension $\geq 5$ on $C$ such that
the image of the multiplication map
$$W'_3\otimes W'_5\rightarrow H^0(C,\mathcal{L}_C(1))$$
has rank $\leq 8$. Up to taking a general vector subspace, we can assume ${\rm dim}\,W'_5=5$.
Let $x,\,y,\,z$ be three general points of $C$. Then the linear system $W'_{2,x,y,z}$ of elements of $W'_5$ vanishing on $x$, $y$ and $z$ has dimension $2$ and the rank of the multiplication map
$$W'_3\otimes W'_{2,x,y,z}\rightarrow H^0(C,\mathcal{L}_C(1)(-x-y-z))$$
is at most $5$, hence has a nontrivial kernel. By the base-point-free pencil trick, one has
$H^0(C,\mathcal{L}_C^{-1}(x+y+z)(1))\not=0$, hence ${\rm deg}\,\mathcal{L}_C^{-1}(x+y+z)(1)>0$ since $x,\,y,\,z$ are arbitrary.
It follows that  ${\rm deg}\,\mathcal{L}_C<6$, hence ${\rm deg}\,\mathcal{L}_C=5$ and the linear system $W'_5$ is complete.
\end{proof}
We now conclude the proof of Lemma \ref{le2pourcubicdim4}.  As the rational map $\phi_{2L+M}$ on each curve $D\subset X$ as above factors through the corresponding curve $C\subset Y$, there exists a line bundle  $\mathcal{L}$  on the chosen  desingularization $\widetilde{Y}$ of $Y$, such that
$|\mathcal{L}|$ has no fixed components and
$\phi_{\mathcal{L}}\circ \phi_{L+M}=\phi_{2L+M}$.  As we already explained in the proof of Lemma \ref{le3sur3dim4}, the $10$-dimensional linear system $H^0(X,2L+M)$   has no fixed component. This implies
  formula
 (\ref{eqpourpullbacl}), where the divisor $E$ appears because a divisor contracted by $\phi_{L+M}$ can appear in the fixed part of the
 linear system $\phi_{L+M}^*|\mathcal{L}|$. Equality (\ref{eqpourpullbaclsec})  follows. Finally, as $H^0(\widetilde{Y},\mathcal{L})$ has no fixed component and $C\subset \widetilde{Y}$ is in general position, $H^0(\widetilde{Y},\mathcal{L})_{\mid C}$ has no base-point, hence, $\mathcal{L}_{\mid C}=\mathcal{L}_C$, where $\mathcal{L}_C$ appears in Claim \ref{claimpourCetdeg8}. It thus follows from Claim  \ref{claimpourCetdeg8} that ${\rm deg}\,\mathcal{L}_{\mid C}=5$.
\end{proof}
Lemma \ref{le2pourcubicdim4} indicates that if $Y$ is a cubic hypersurface, it has a singular locus which is of dimension at least $1$. Indeed, if ${\rm Sing}\,Y$ is isolated, the general hyperplane section $Y'$  of $Y$ is smooth, hence has Picard number $1$ and thus any line bundle on $Y_{\rm reg}$ has degree divisible by $3$ on the plane sections of $Y'$. Going farther, we  now prove
\begin{lemm}  \label{le3pourcubicdim4} If $Y$ is a cubic hypersurface, the singular locus of $Y$ has dimension at least $2$.
\end{lemm}
\begin{proof}  Assume by contradiction that the singular locus of the cubic hypersurface $Y$ has  dimension $\leq1$. The notation $\mathcal{L}_{\rm reg}\in {\rm Pic}\,Y_{\rm reg}$ being as in Lemma \ref{le2pourcubicdim4}, we  prove
\begin{claim}\label{claimpourP3} There exists a divisor $D \subset Y$ which is a linear $\mathbb{P}^3\subset Y\subset \mathbb{P}^5$ such that
\begin{eqnarray}\label{eqLP3} \mathcal{L}_{\rm reg}=\mathcal{O}_{Y_{reg}}(2)(-D).
\end{eqnarray}
\end{claim}
\begin{proof} Let $S\subset H\subset Y$ be  a general $2$-dimensional (respectively 3-dimensional) linear section of $Y$. The surface $S$ is thus smooth by our assumption and contained in $Y_{\rm reg}$. Assume $c_1(\mathcal{L}_{\mid S})^2\geq 7$. Then, denoting $\mathcal{L}':=\mathcal{O}_{Y_{\rm reg}}(2)\otimes  \mathcal{L}_{\rm reg}^{-1}$,  the line bundle
$\mathcal{L}'_S:=\mathcal{L}'_{\mid S}$ satisfies
\begin{eqnarray}\label{eqnombresdint}
c_1(\mathcal{L}'_S)\cdot K_S=-1,\,c_1(\mathcal{L}'_S)^2\geq 7+12-20=-1,
\end{eqnarray}
hence we have
$\chi(S,\mathcal{L}'_S)\geq 1$ and it  follows from the first equality in (\ref{eqnombresdint})  that $h^2(S,\mathcal{L}'_S)=0$, hence $h^0(S,\mathcal{L}'_S)\not=0$. Thus
$\mathcal{L}'_S=\mathcal{O}_S(\Delta_S)$ for some line $\Delta_S\subset S$.
We now show that the   lines $\Delta_S$  for  all $S\subset Y_{\rm reg}$ fill-in only a  divisor $D$  in $Y_{\rm reg}$. To see this, we observe that
if $C\subset S$ is a smooth plane section, the intersection
$\Delta_S\cap C$ is a point $x\in C$ that satisfies
$$\mathcal{O}_C(x)=\mathcal{L}'_{\mid C}.$$
It follows that $x$ does not depend on $S$ containing $C$.  Fixing $x$ and moving $C$ containing $x$, we finally conclude that, if a point $x\in \Delta_S$ for some $S$, then $x\in \Delta_S$ for any $S$ containing $x$. This proves the existence of the divisor $D$.
This divisor is then  a $\mathbb{P}^3$ since it contains at least a $4$-dimensional family of lines (indeed, a given line is contained in  a $4$-dimensional family of surfaces $S$ and there is a $8$-dimensional family of surfaces $S$).
Finally, we found that  the divisor $D\cong \mathbb{P}^3\subset Y$ satisfies
$$\mathcal{L}_{\mid S}=\mathcal{I}_{D}(2)_{\mid S}$$
for any smooth surface $S\subset Y_{\rm reg}$.  It easily  follows that
$\mathcal{L}_{\rm reg}=\mathcal{I}_{D}(2)$.

We next assume that  $c_1(\mathcal{L}_{\mid S})^2\leq 5$. Then, as $|\mathcal{L}_{\mid S}|$ has no fixed part, it is nef and we have
$h^1(S,\mathcal{L}_{\mid S})=0,\,h^2(S,\mathcal{L}_{\mid S})=0$, since $-K_S$ is ample. Hence we have in this case
\begin{eqnarray}\label{eqhzeroLS} h^0(S,\mathcal{L}_{\mid S})=1+\frac{c_1(\mathcal{L}_{\mid S})^2-c_1(\mathcal{L}_{\mid S})\cdot K_S }{2}\leq 6.
\end{eqnarray}
Comparing with (\ref{eqpourpullbaclsec}), we get $h^0(Y_{\rm reg},\mathcal{L}_{\rm reg}\otimes \mathcal{I}_S)\geq 4$, and considering as above a general pair
$S\subset H\subset Y$, we conclude that either

(1)
$h^0(Y_{\rm reg},\mathcal{L}_{\rm reg}(-1))\geq2$, or

(2)
$h^0(H_{\rm reg},\mathcal{L}_{\rm reg}(-1)_{\mid H_{\rm reg}})\geq3$.

In  case (1), the divisor of a general  section of $\mathcal{L}_{\rm reg}(-1)$ has degree $2$, hence it is reduced by Bertini and  there are two possibilities:

(a)  The divisor of this section is the intersection $Q_3\cap Y_{\rm reg}$ for some  $3$-dimensional quadric $Q_3\subset Y$ such that
$Q_3\cap Y_{\rm reg} \in |\mathcal{L}_{\rm reg}(-1)|$. There is then a residual $\mathbb{P}^3\subset Y$ such that
$\mathbb{P}^3+Q_3$ is a hyperplane section of $Y$ and the lemma is  also proved in this case.

(b)  The divisor of this section is  the intersection with $Y_{\rm reg}$ of the  union $D_1\cup D_2$ of two $\mathbb{P}^3$'s contained in $ Y$ such that
$(D_1\cup D_2)\cap Y_{\rm reg} \in |\mathcal{L}(-1)|$ and ${\rm dim}\,D_1\cap D_2=1$. In fact, this case is impossible because the intersection of $D_1\cup D_2$ with a general  cubic surface $S\subset Y$ as above is the disjoint union of two lines, hence is a rigid divisor. It follows immediately that $h^0(Y_{\rm reg},\mathcal{L}_{\rm reg}(-1))\leq1$, contradicting the inequality   (1).

The  case (2) is excluded as follows. We get that the divisor of a general  section of $\mathcal{L}_{\rm reg}(-1)_{\mid H_{\rm reg}}$ has degree $2$, hence it is reduced by Bertini and  there are two possibilities:

(a)
There is a $2$-dimensional quadric $Q_2\subset H$ such that
$Q_2\cap H_{\rm reg}\in |\mathcal{L}_{\mid H_{\rm reg}}(-1)|$. But then $h^0(H_{\rm reg},\mathcal{L}_{\rm reg}(-1)_{\mid H_{\rm reg}})\leq2$, contradicting the inequality (2).

(b)  There are two planes  $P,\,P'\subset H$ such that
$(P+P')\cap H_{\rm reg}\in |\mathcal{L}_{\mid H_{\rm reg}}(-1)|$. But then, we find as above $h^0(H_{\rm reg},\mathcal{L}(-1)_{\mid H_{\rm reg}})\leq1$, contradicting the inequality  (2).

The claim is thus proved.
\end{proof}
We now conclude the proof of Lemma \ref{le3pourcubicdim4}.  Let $D\cong \mathbb{P}^3\subset Y$ be as in Claim \ref{claimpourP3}. We have
$h^0(Y_{\rm reg}, \mathcal{I}_D(2))=11$, while $h^0(X,2L+M)=10$.
The inclusion $H^0(X,2L+M)\subset H^0(Y_{\rm reg}, \mathcal{I}_D(2))$ given by Lemmas \ref{le2pourcubicdim4} and  Claim \ref{claimpourP3} is thus the inclusion of a hyperplane.  Let $H_X\subset X$ be a general member of $|L+M|$, that is, the inverse image
$\phi_{L+M}^{-1}(H_Y)$ where $H_Y\subset Y$ is a general hyperplane section. Then
$H^0(Y_{\rm reg}, \mathcal{I}_{H_Y}\otimes \mathcal{I}_D(2))=H^0(Y_{\rm reg},  \mathcal{I}_D(1))$ has dimension $2$, hence it intersects nontrivially the hyperplane $H^0(X,2L+M)\subset H^0(Y_{\rm reg}, \mathcal{I}_D(2))$, providing a nonzero section of the line bundle $$(2L+M)\otimes \mathcal{I}_{H_X}=L$$
on $X$, which  is excluded by the  hypotheses of  Theorem  \ref{propLMnef3}. The lemma is thus proved.
\end{proof}

\begin{proof}[Proof of Proposition \ref{procubquatre}]  Using Lemmas \ref{le3pourcubicdim4} and \ref{le1pourcubicdim4}, the singular locus of a cubic hypersurface $Y={\rm Im}\,\phi$ has dimension $2$. We observe now that the arguments  in  Lemma \ref{le3pourcubicdim4} involving smooth cubic surfaces appearing as general linear sections of $Y$ when ${\rm dim}\,({\rm Sing}\,Y)\leq 1$ extend in a straightforward way  if the general cubic surface section has Duval singularities, which happens if  the order of vanishing of  the defining equation $f_Y$ of $Y$  along any $2$-dimensional component of its singular locus is not $3$ (see \cite{cheltsov}). Indeed, we can    work in that case  with a crepant resolution of singularities of these surfaces.  However, if ${\rm dim}\,({\rm Sing}\,Y)=2$ and $f_Y$ vanishes to order $3$ along a component of ${\rm Sing}\,Y$,   $Y$ is a cone over an elliptic curve in $\mathbb{P}^2$. This case is excluded since $Y$ would then have many reducible hyperplane sections. This concludes the proof of Proposition \ref{procubquatre}.
\end{proof}
\begin{prop}\label{proquatrecinqtrois} Let $X,\, L,\, M$ be as above. Then the image $Y=\phi_{L+M}(X)$ cannot be  a linearly nondegenerate threefold of degree $4$ or $5$ in $ \mathbb{P}^5$.
\end{prop}

Assuming by contradiction that $Y$ is a threefold of degree $4$ or $5$, we first prove the following    lemmas.
\begin{lemm} \label{leserie2A} The threefold  $Y$ cannot be a cone $\pi:Y\dashrightarrow S$ over a surface $S$ in $\mathbb{P}^4$.
\end{lemm}
\begin{proof} As in the proof of Lemma \ref{leeasyA}, this would indeed contradict Lemma \ref{nolmsur2crucial} by considering the composite map $\pi\circ \phi_{L+M}:X\dashrightarrow S$.
\end{proof}
\begin{lemm} \label{leserie2AA} The threefold $Y\subset \mathbb{P}^5$ is not contained in a quadric of rank $\leq 4$.
\end{lemm}
\begin{proof} Indeed,  $Y$ would have otherwise many reducible hyperplane sections, contradicting the fact that all members of $|L+M|$ are irreducible.
\end{proof}
We now exclude the case of degree $4$.
\begin{lemm} \label{leserie2Bbis} The threefold $Y$ cannot be of degree $4$.
\end{lemm}
\begin{proof} As $Y$ is not a  cone and  is linearly nondegenerate, linearly normal in $\mathbb{P}^5$, the Swinnerton-Dyer classification \cite{SW} tells us that $Y$ is the complete intersection of two quadrics in $\mathbb{P}^5$. In particular, $Y$ contains a line through any of its points. Furthermore, if   $Y$ is smooth, its family of lines is smooth and connected and $Y$ contains  4 lines through a general point $y\in Y$. Let $\Delta_1,\,\ldots,\,\Delta_4$ be  the four lines through $y$. Then the $\Delta_i$ are contained in the projectivized tangent space $\mathbb{P}^3_y$ of $Y$ at $y$, and   $\mathbb{P}^3_y\cap Y\supseteqq \cup_i\Delta_i$, while by smoothness of $Y$, $\mathbb{P}^3_y\cap Y$ has dimension $1$; hence we have in fact
$\mathbb{P}^3_y\cap Y= \cup_i\Delta_i$. The     inverse images $S_i:=\phi_{L+M}^{-1}(\Delta_i)$ are then cohomologous in $X$ and their common class $f$ satisfies
$$4f+e=(l+m)^2$$
for some pseudoeffective class in $X$. This contradicts again Lemma \ref{nolmsur2crucial}. We  now consider the case where  $Y$ is singular and try to extend the argument above. We still  know that $Y$ is swept-out by lines and that there exist at least 4 lines passing through a general point of $Y$. Unfortunately we do not know that the lines are homologous or algebraically equivalent in $Y$, so the above argument fails. However, we have the following
\begin{sublemm}\label{sublemmnomtardif}  If $Y={\rm Im}\,\phi_{L+M}$ is the intersection of two quadrics in $\mathbb{P}^5$, there are at most two algebraic equivalence classes of mobile lines in $Y$.
\end{sublemm}
\begin{proof}   We first claim that the family of conics in $Y$ has at most two irreducible 4-dimensional  components whose general point parameterizes a conic  passing  through the general point of $Y$. Indeed, $Y$ is not swept-out by planes, otherwise it has many reducible hyperplane sections, which is excluded. Hence we can consider only the family of conics in $Y$ which are not contained in a plane contained in $Y$. But these conics are in bijection with planes contained in one quadric $Q_t$ containing $Y$ and not contained in $Y$.  By Lemma \ref{leserie2AA}, $Y$ is not contained in any quadric of rank $\leq 4$. The family of planes in $Q_t$ thus  has two irreducible components of dimension $3$ if $Q_t$ is smooth and only one, also of dimension 3, if $Q_t$ is singular of rank $5$. Thus the family of planes contained in one of the  $Q_t$ has one or  two components, according to whether the double cover of the projective line parameterizing the quadrics $Q_t$ containing $Y$ determined by the choice of a rulling is reducible or not. This proves the claim. Let now $y\in Y$ be  a general point. There are (at least) $4$ lines $l_1,\ldots,\,l_4$ in $Y$ passing through $y$, and the union of any two of these lines is a conic in $Y$ passing through $y$. Hence the cycles $l_i+l_j$ belong to only two algebraic equivalence classes of $1$-cycles in $Y$, and it follows immediately that these four lines  belong to at most  two algebraic equivalence classes of 1-cycles in $Y$.
\end{proof}
\begin{coro}\label{coropour3foldcase} (i) There exists an algebraic equivalence class $\mathcal{C}$ of $1$-cycles on $Y$ such that, through  a general point $y\in Y$, there passes at least two lines of the class $\mathcal{C}$.

(ii)  There exist chains of three lines $\Delta_1,\,\Delta_2,\,\Delta_3\subset Y$, $\Delta_1\cap\Delta_2\not=\emptyset,\,\Delta_2\cap\Delta_3\not=\emptyset$, such that the $\Delta_i$ pass through the general point of $Y$ and the three lines $\Delta_i$ are in the class $\mathcal{C}$.
\end{coro}
\begin{proof}   Statement (i) immediately follows from Sublemma \ref{sublemmnomtardif}, since there passes four lines through a general point of $Y$.

(ii) Given a general   point $y$ of $Y$, there are two lines $\Delta_1,\,\Delta_2$ in the algebraic equivalence class $\mathcal{C}$ and passing through $y$. Choosing another point $y'\in\Delta_2$, we can choose a deformation $\Delta_3$ of $\Delta_1$ (hence also in the class $\mathcal{C}$) passing through $y'$. This gives the desired chain.
\end{proof}
 Corollary \ref{coropour3foldcase} leads to a contradiction  as follows: indeed the three lines $\Delta_i$ forming a connected chain  are all contained in a mobile  $\mathbb{P}^3_{\Delta_\cdot}$.  Assume first  that $\mathbb{P}^3_{\Delta_\cdot}\cap Y$ is $1$-dimensional; then
 we get, by taking inverse images in  $X$, an equality of codimension $2$ cycles in $X$
\begin{eqnarray}\label{eqtroisdroites} \phi^{*}_{L+M}(\mathbb{P}^3_{\Delta_\cdot}\cap Y)=T_{1}+T_{2}+T_{3}+ T\,\,{\rm in}\,\,A^2(X)
\end{eqnarray}
where $T$ is the class of an effective surface in $X$ and $T_i=\phi_{L+M}^{-1}(\Delta_i)$.
As the three lines $\Delta_i$ are algebraically equivalent in $Y$, the three surfaces $T_i$ are numerically equivalent in $X$, and thus (\ref{eqtroisdroites}) contradicts Lemma \ref{nolmsur2crucial}.
It remains to analyze the case where $\mathbb{P}^3_{\Delta_\cdot}\cap Y$ has a $2$-dimensional component for general $\Delta_\cdot$. If this component is mobile, then $Y$ has many reducible hyperplane sections, which is excluded. If this component is fixed, it must be a plane $P\subset Y$ since it is contained in the intersection of at least two $\mathbb{P}^3$'s, and this plane has   the property that any mobile  line in $Y$ in the algebraic equivalence class $\mathcal{C}$   intersects $P$. In that case, under  the linear projection $\pi_P:Y\dashrightarrow \mathbb{P}^2$ from $P$, $Y$ maps to a curve  of degree $>1$ in  $\mathbb{P}^2$, since the fiber of the projection $\pi_P$ passing through the general point $y$ contains at least two lines, hence must be of dimension at least $2$.  Hence $Y$ has many reducible hyperplane sections, which gives  again a contradiction.
Lemma \ref{leserie2Bbis} is thus proved, hence also Proposition \ref{proquatrecinqtrois} in the case of degree $4$.
\end{proof}

\begin{proof}[Proof of Proposition \ref{proquatrecinqtrois}] By Lemma \ref{leserie2Bbis}, we only have to exclude the case where $Y$ is a threefold of  degree $5$.

\begin{claim} \label{leserie2Bnewnew} $Y$ is not  contained in a quadric.
\end{claim}

\begin{proof} If  $Y$ is contained in a quadric $Q\subset \mathbb{P}^5$,   $Q$ must have  rank at least $5$  by Lemma \ref{leserie2AA}. The  general hyperplane section $Q_H:=Q\cap H$ of $Q$ is then a  smooth quadric of dimension $3$ which contains a surface of degree $5$, contradicting the fact that ${\rm Pic}\,Q_H$  is generated by $\mathcal{O}_Q(1)$.
\end{proof}
Denote by $n: Y_n\rightarrow Y$ the normalization of $Y$. Thus  $Y_n$ is smooth in codimension $1$. For a general $\mathbb{P}^3\subset \mathbb{P}^5$, the general section $C_n:=n^{-1}(C),\,\,C:=Y\cap\mathbb{P}^3$,  of $Y_n$   is a smooth connected  curve. We denote by $S$ a general hyperplane section of $Y$, $S_n\subset Y_n$ its inverse image in $Y_n$ and consider the inclusions  $C_n\subset S_n\subset Y_n$.
\begin{claim} \label{leserie2Bnewnewnew} The curve $C\subset \mathbb{P}^3$ is a smooth genus $2$ curve of degree $5$ (in particular it is isomorphic to $C_n$).  It  is thus contained in a quadric.
\end{claim}
\begin{proof} The second statement follows from the first since it implies that $h^0(C,\mathcal{O}_C(2))=9$.
Let  $\tau: \widetilde{Y}\rightarrow Y_n$ be a desingularization, and let $\tau':=n\circ \tau$.
We have
\begin{eqnarray}\label{eqnewdesections} H^0(\widetilde{Y}, {\tau'}^*\mathcal{O}_{{Y}}(1))\cong H^0(Y_n,n^*\mathcal{O}_{{Y}}(1))\cong  H^0(Y,\mathcal{O}_Y(1)).\end{eqnarray} Indeed, via the dominant  rational map
$${\tau'}^{-1}\circ \phi_{L+M}: X\dashrightarrow \widetilde{Y},$$
sections of ${\tau'}^*\mathcal{O}_{{Y}}(1)$ on $\widetilde{Y}$ pull-back to sections of $L+M$ on $X$, while by construction, the pull-back map  $\phi_{L+M}^*: H^0(Y,\mathcal{O}_Y(1))\rightarrow  H^0(X,L+M)$ is an isomorphism.
We have \begin{eqnarray}\label{eqvanishingnewdu23} H^1(\widetilde{Y},\mathcal{O}_{\widetilde{Y}})=0\end{eqnarray} since $\widetilde{Y}$ is rationally dominated by $X$. Let $\widetilde{S}:={\tau'}^{-1}(S)\subset \widetilde{Y}$, so that $\widetilde{S}\in |{\tau'}^*\mathcal{O}_{{Y}}(1)|$. By the vanishing (\ref{eqvanishingnewdu23}) and using (\ref{eqnewdesections}), we conclude that $h^0(\widetilde{S}, {\tau'}^*\mathcal{O}_{{S}}(1))=5$.
Furthermore, we also have $H^1(\widetilde{S},\mathcal{O}_{\widetilde{S}})=0$, using (\ref{eqvanishingnewdu23}) and the fact that
$\widetilde{S}\subset \widetilde{Y}$ is big and nef. As $C_n\subset \widetilde{S}$ belongs to $|{\tau'}^*\mathcal{O}_{S}(1)|$, we conclude as before that, denoting $\mathcal{O}_{{C_n}}(1):=  {\tau'}^*\mathcal{O}_{C}(1)$,
$$h^0(C_n, \mathcal{O}_{{C_n}}(1))=h^0(\widetilde{S}, {\tau'}^*\mathcal{O}_{{S}}(1))-1=4,$$
which implies that $h^1(C_n, \mathcal{O}_{{C_n}}(1))=0$, hence  $g(C_n)=2$,    since we know that the line bundle $\mathcal{O}_{{C_n}}(1)$ on $C_n$ has degree $5$. It follows that  $\mathcal{O}_{C_n}(1)$ is very ample on $C_n$, hence $C_n$ is isomorphic to $C$.
\end{proof}
We get a contradiction from Claims \ref{leserie2Bnewnew} and \ref{leserie2Bnewnewnew} by observing that both restriction maps
$$H^0(\mathbb{P}^5,\mathcal{I}_Y(2))\rightarrow H^0(\mathbb{P}^4,\mathcal{I}_S(2)),$$
$$H^0(\mathbb{P}^4,\mathcal{I}_S(2)) \rightarrow H^0(\mathbb{P}^3,\mathcal{I}_C(2))$$
are  surjective. Indeed, the surjectivity in both cases is implied by the respective vanishings
$H^1(\mathbb{P}^5,\mathcal{I}_Y(1))=0$ and $H^1(\mathbb{P}^4,\mathcal{I}_S(1))=0$, that come from the fact that  both $Y\subset \mathbb{P}^5$ and $S\subset \mathbb{P}^4$ are linearly normal (for the surface $S$, this follows indeed from the arguments given in the previous proof). This concludes the proof of Proposition \ref{proquatrecinqtrois}.
\end{proof}

    \end{document}